\documentclass[12pt]{article}
\usepackage{amsmath,amssymb}
\usepackage{amsthm}

\def\LC{\mathcal{L}}

\def\HC{\mathcal{H}}

\def\E{\mathbf{E}}

\def\R{\mathbf{R}}
\def\Z{\mathbf{Z}}

\def\1{\mathbf{1}}

\def\al{\alpha}
\def\be{\beta}
\def\pa{\partial}
\def\ep{\epsilon}
\def\de{\delta}

\newtheorem{prop}{Proposition}[section]
\newtheorem{theorem}{Theorem}[section]

\newtheorem{corollary}{Corollary}
\newtheorem{remark}{Remark}

\newcommand{\la}{\lambda}
\newcommand{\si}{\sigma}

\newcommand{\Ga}{\Gamma}
\newcommand{\De}{\Delta}

\begin{document}
\title{Chronological operator-valued Feynman-Kac formulae for generalized fractional evolutions}

\author{Vassili N. Kolokoltsov\thanks{Department of Statistics, University of Warwick,
 Coventry CV4 7AL UK, associate member of IPI RAN Moscow,
  Email: v.kolokoltsov@warwick.ac.uk}}
\maketitle

\begin{abstract}
We study the generalized fractional linear problem $D^{\nu}_{a+*} f(x) =A(x)f(x)+g(x)$, where $D^{\nu}$ is an arbitrary mixture of Caputo derivatives
of order at most one and $A(x)$ a family of operators in a Banach space generating strongly continuous semigroups. For time homogeneous case,
when $A(x)$ does not depend on time $x$, the solution is expressed by the generalized operator-valued Mittag-Leffler function.
For the more involved time-dependent case we use the method of non-commutative operator-valued Feynman-Kac formula in combination
with the probabilistic interpretation of Caputo derivatives suggested recently by the author to find the general integral representation
of the solutions, which are new even for the case of the standard Caputo derivative $D^{\be}_{a+*}$.
In the point of view adopted here we analyse the fractional equations not as some 'exotic evolutions',
but rather as 'standard' stationary problems leading to the stationary non-commutative operator-valued Feynman-Kac representation.

\end{abstract}



\section{Introduction}

Recall that a backward propagator in a Banach space $B$ (also referred to in the literature as evolutionary family
or two-parameter semigroup) is a family of bounded linear operators $U^{t,s}$, $t\le s$, such that $U^{t,t}$ is the identity for any $t$
 and the chain rule $U^{s,t}U^{t,r}=U^{s,r}$ holds for all $s\le t \ge r$. Such propagator is called strongly continuous if
 $U^{t,s}f$ is a continuous function of $t$ and $s$ for any $f\in B$, in which case, as follows from the principle of uniform boudedness,
it is locally uniformly bounded: $\|U^{t,s}\| \le C$ for $t,s$ from any compact set.

Let $B$ be a Banach space and $D$ its dense subspace. Let $A_t$ be a family of linear operators $D\to B$ and $U^{t,s}$ a strongly continuous
backward propagator in $B$ with $D$ being invariant under all $U^{t,s}$. We say that $\{A_t\}$ generates $U^{t,s}$ on $D$, if for any $f\in D$,
the equations
\begin{equation}
\label{defpropgene}
\frac{d}{dt} U^{t,s} f=-A_t U^{t,s}
\end{equation}
hold. In particular, $U^{t,s}$ is the resolving operator for the Cauchy problem
\begin{equation}
\label{eqdefgenCauchy}
\frac{d}{dt} f_t=-A_t f_t, \quad f_s=f, \quad t\le s.
\end{equation}

In case of commuting bounded operator $A_t$, the propagator $U^{t,s}$ can be  expressed as the exponent
\begin{equation}
\label{eqdefgenCauchyexp}
U^{t,s}=\exp\{ \int_t^s A_{\tau} \, d\tau \}
\end{equation}

In case of non-commuting $A_t$ this formula does not hold and the correctly modified expression is referred to as the
(backward) chronological exponential:
\begin{equation}
\label{eqdefgenCauchychrexp}
U^{t,s}=T \exp\{ \int_t^s A_{\tau} \, d\tau \}=\lim_{|\De|\to 0}  \exp\{(t_1-t_0) A_{t_0}\} \cdots \exp\{ (t_n-t_{n-1})A_{t_{n-1}}\},
\end{equation}
where $\De =\{s=t_0 <t_1 <\cdots < t_n=t\}$ is a partition of the integral $[s,t]$ and $|\De|=\max_j(t_j-t_{j-1})$.

\begin{remark} In the forward exponential that we are not considering here the order of exponents is reversed.
\end{remark}

Extending the notation used for the case of bounded $A_t$, the propagator $U^{t,s}$ is expressed as the chronological exponent
$T \exp\{ \int_t^s A_{\tau} \, d\tau \}$. This expression can be considered as a customary notation, but in fact in many cases
one can show the convergence of the r.h.s. of \eqref{eqdefgenCauchychrexp} even if $A_t$ are unbounded.

The fractional analog of the Cauchy problem \eqref{eqdefgenCauchy} represents the problem
\begin{equation}
\label{eqdefgenCauchyfrac}
D^{\be}_{a+*} f(x)=A(x) f(x)+g(x), \quad f_a=Y, \quad x\ge a,
\end{equation}
with
\begin{equation}
\label{eqCder0}
D^{\be}_{a+*}f(x)
=\frac{1}{\Ga (-\be)} \int_0^{x-a}\frac{f(x-z)-f(x)}{z^{1+\be}}dz
+\frac{f(x)-f(a)}{\Ga (1-\be) (x-a)^{\be}},
\end{equation}
denoting the Caputo derivative of order $\be \in (0,1)$. Though this is not the most standard definition of Caputo derivative,
it is also well known and derivable from the standard one by straightforward manipulations with partial integration.

There is quite a lot of research dealing with the equations where $D^{\be}_{a+*}$
is replaced by the mixture (discrete or even continuous) of the derivatives with different values of $\be$.
Extending these ideas led to the development of the generalized fractional calculus.
Usually it is developed by extending fractional integrals to the integral operators with arbitrary integral kernels
(or some of their subclasses) and then defining the fractional derivatives as the derivatives of these integral operators,
see \cite{AgrawGenerCalcVar10}, \cite{MalOdTor15}, \cite{Kir94}, \cite{Kir08}, \cite{Kir14}, \cite{Kochubei11}.
However, thus defined fractional derivative is not the inverse operator to the fractional integral.
In \cite{Ko15} a different approach to the generalized fractional operations was
suggested. Motivated by probabilistic interpretation (L\'evy processes interrupted on the attempts to cross the boundary), it starts
with the definition of a generalized fractional derivative, and the corresponding generalized fractional integral is then defined as the
corresponding potential operator, or in other words, as the right inverse operator to the fractional derivative, which represents
the integral operator with the integral kernel being the fundamental solution of the operator of generalized fractional derivative.
We follow this approach here. Namely, taking into account that the operator $-D^{\be}_{a+*}$ is the generator of the
time-inverted $\be$-stable L\'evy subordinator restricted to the space of functions that are constant to the left of $a$, it
was suggested in \cite{Ko15} that a natural general framework to deal with all extensions is to replace $-D^{\be}_{a+*}$ by the generator
of an arbitrary time-inverted L\'evy subordinator
\begin{equation}
\label{eqlevysubinv}
L'_{\nu}f(x)=\int_0^{\infty} (f(x-y)-f(x))\, \nu(dy),
\end{equation}
with $\int \min(1,y) \nu(dy) <\infty$, also restricted to the space of functions that are constant to
 the left of $a$. This leads to the extension of $D^{\be}$ in the form
\begin{equation}
\label{eqdefmixedfracderCap0}
D^{(\nu)}_{a+*}=-\int_0^{x-a} (f(x-y)-f(x))\nu (dy) -\int_{x-a}^{\infty} (f(a)-f(x))\nu (dy),
\end{equation}
and to the corresponding extension of equation \eqref{eqdefgenCauchyfrac}:
\begin{equation}
\label{eqdefgenCauchyfracgen}
D^{(\nu)}_{a+*} f(x)=A(x)f(x)+g(x), \quad f_a=Y, \quad x\ge a.
\end{equation}

The objective of this paper is two-folds. Firstly, we derive the solution
to problem \eqref{eqdefgenCauchyfracgen} for $A(x)$ not depending on $x$ in terms for arbitrary $\nu$ leading to the introduction of
 the generalized operator-valued Mittag-Leffler function. For simpler case of vanishing $g$, it writes down as
\[
f(x)=E_{(\nu),x-a}(A)Y,
\]
where the generalized operator-valued Mittag-Leffler function is defined as
\[
E_{(\nu),z}(A)= \int_0^{\infty} e^{ tA}  \frac{\pa}{\pa t}\left(  \int_z^{\infty} G_{(\nu)}(t,dy)\right) \, dt
\]
in terms of the transition density $ G_{(\nu)}(t,dy)$ of the subordinator generated by $L'_{\nu}$,
extending the corresponding well-known integral representation of the classical Mittag-Leffler function in terms of the stable densities.
This extension of the Mittag-Leffler functions is different from the extensions introduced by various authors in the context of more concrete
equations (see \cite{Kir14}) and represents a version of the extension introduced in \cite{HerHer16} in a similar context.

 The second objective is to prove the following formula for the solution of general problem
\eqref{eqdefgenCauchyfracgen} (valid under some technical assumptions on $A(x)$) in terms of the backward chronological exponential:
\begin{equation}
\label{eqtimeordFKfost0}
f(x)=Y+\E_{\nu} \int_0^{\si_a} [ T\exp\{\int_0^s A(Z_x(\tau)) \, d\tau \}(A(Z_x(t))Y+g(Z_x(t)))] \, ds,
\end{equation}
where $\E_{\nu}$ is the expectation on the paths $Z_x(t)$ of the L\'evy process generated by $L'_{\nu}$ and started at $x$,
and $\si_a$ is a stopping time, when the path $Z_x(t)$ exits the half-line $(a,\infty)$.
This formula can be looked at as the stationary version of time-ordered operator-valued Feynman-Kac formula.
By passing the Cauchy problem generated by $A(x)- D^{(\nu)}_{a+*}$ is analyzed,
whose solution is presented in terms of the corresponding non-stationary  time-ordered operator-valued Feynman-Kac formula.

At the end of the paper various particular classes of equations are presented (Schr\"odinger equations, generalized diffusions,
equations with spatially homogeneous pseudo-differential generators) that fit into the general scheme developed.

Needless to say that the general integral representations for the solutions developed above are particulary suitable both for
the development of the effective numeric schemes of approximate calculations and for the development of the theory of nonlinear
 equations, the solutions to the latter being constructed as the fixed points of integral operators.


The current literature on fractional calculus, and in particular on the equations of type \eqref{eqdefgenCauchyfrac},
is enormous, so that it is hard to present a reasonably brief review even of the papers devoted to the various kinds
of concrete equations incorporated in the general abstract framework \eqref{eqdefgenCauchyfracgen}. Therefore we refer
to some recent books on the subject and references therein: \cite{Scalabook}, \cite{KiSrTr06}, \cite{Podlub99}, \cite{Dietbook},
\cite{Umarovbook15}, \cite{PskhuRusBook}, \cite{Libook15}, \cite{HendRodBook16}, see additionally \cite{AtDoPiSt},
\cite{ChenMeNa12}, \cite{GarGuiMaiPag}, \cite{GorLuchSto}, \cite{Kochubei11}, \cite{KochKondr17}, \cite{leonmeersik13},
\cite{Pskhu11}, \cite{Scal12} for some related recent developments. We refer to books \cite{LorFeKaBook11}, \cite{DelMorbook},
\cite{GulCasbook} and extensive references therein for the general background on the Feynman-Kac formulae, and to \cite{Hagood},
\cite{Jeffr96}, \cite{Jeffr96book} for the introduction to their non-commutative versions. Here we are working with different
settings and present an independent development via the compound Poisson process approximation (finite $\nu$), where the direct
analytic construction of the path integral is available.

The paper is organized as follows. In the next two sections some tools of analysis are introduced, namely, potential measures and path integrals.
In Section \ref{secgenfrac} the generalized fractional calculus in the spirit of \cite{Ko15} are motivated and properly introduced.
In Sections \ref{secgenfraclin} and \ref{secgenfracasym} the time-homogeneous equations are analyzed leading to the solutions in terms
of the generalized operator-valued Mittag-Leffler functions. In sections \ref{secgenfractimedep} - \ref{secexam} the main results
concerning formula \eqref{eqtimeordFKfost0} are obtained and examples are presented.

\section{Preliminaries: vector-valued convolution semigroups}

Recall that the evolution equation governing the L\'evy subordinators  have the form
\begin{equation}
\label{eqlevysubo}
\dot f_t =L_{\nu}f(x)=\int_0^{\infty} (f(x+y)-f(x))\nu(dy)=\int_{-\infty}^{\infty} (f(x+y)-f(x))\nu(dy),
\end{equation}
where $\nu$ is a measure on $\R$ such that $\nu((-\infty,0])=0$ and satisfying the {\it one-sided L\'evy condition}
\begin{equation}
\label{eqalevysubo}
\int_0^{\infty} \min(1,y)\nu(dy)<\infty.
\end{equation}
 It is well known that equations \eqref{eqlevysubo} generate the Feller
 semigroups $T_t$ on $C_{\infty}(\R)$ such that
\begin{equation}
\label{eq1propLevyFellersub}
T_tf(x)=\int_0^{\infty} f(x+y)G_{(\nu)}(t,dy)=\int_{-\infty}^0 f(x-y)\tilde G_{(\nu)}(t,dy),
\end{equation}
with some probability measures $G_{(\nu)}(t,dy)$ on $\R_+$ and $\tilde G_{(\nu)}(t,dy)=G_{(\nu)}(t,d(-y)$ on $\R_-$,
so that the value of $T_tf(x)$ depends only on $f(z)$ with $z\ge x$.
The space $C^1_{\infty}(\R)$ represents an invariant core for $T_t$.
The semigroup property of $T_t$ recast in terms of the probability measures $G_{(\nu)}(t,dy)$
shows that these measures form a semigroup with respect to convolution.
Hence both $T_t$ and $G_{(\nu)}(t,dy)$ are referred to as convolution semigroups.

Symmetrically, the Cauchy problem for the equations
\begin{equation}
\label{eqlevysubo3}
\dot f_t =L'_{\nu}f(x)=\int_0^{\infty} (f(x-y)-f(x))\nu(dy),
\end{equation}
with the generator that is dual to the operator on the r.h.s. of \eqref{eqlevysubo},
have the solutions of the form
 \begin{equation}
\label{eqlevysubo4}
T_t'f(x)=\int_{-\infty}^0 f(x+y)\tilde G_{(\nu)}(t,dy)=\int_0^{\infty} f(x-y) G_{(\nu)}(t,dy).
\end{equation}

Let us stress that $G_{(\nu)}(t,dy)$ is the integral operator of the semigroup $T_t$ generated by $L_{\nu}$, and
it is the Green function of the Cauchy problem \eqref{eqlevysubo3} of the operator $L'_{\nu}$.

It is also known (see e.g. \cite{SchillingBern}) that
for any measure $\nu$  on $\{y:y>0\}$ satisfying \eqref{eqalevysubo} and any $\la \ge 0$
there exist the vague limits
\[
U^{(\nu)}_{\la}(M)=\int_0^{\infty}   e^{-\la t} G_{(\nu)}(t,M) \, dt
\]
of the measures $\int_0^K e^{-\la t}  G_{(\nu)}(t,.) \, dt$, $K\to \infty$, such that $U^{(\nu)}_{\la}(M)$
is finite for any compact $M$.
In particular, for any $z,k>0$
 \begin{equation}
\label{eqestimpotmes}
U^{(\nu)}_0([0,z]) \le e^{kz}/\phi_{\nu}(k),
\end{equation}
where
 \begin{equation}
\label{eqsymbnu}
\phi_{\nu}(\la)=\int_0^{\infty} \int (1-e^{-\la y}) \nu(dy),
\end{equation}
is the Laplace exponent of $L_{\nu}$ and of the corresponding subordinator.
The measure $U^{(\nu)}=U^{(\nu)}_0$ is called the potential measure and
$U^{(\nu)}_{\la}$ the $\la$-potential measure of the L\'evy subordinator or of the convolution semigroup $T_t$.

As the semigroups $T_t$, the potential measures satisfy the following comparison principle:
if $\nu_1(dy)\ge \nu_2(dy)$, then
\begin{equation}
\label{eqcomppotmes}
\int f(y)  U^{(\nu_1)}(dy)\ge \int f(y)  U^{(\nu_2)}(dy)
\end{equation}
for any nondecreasing $f$.

For example, if $\nu$ is finite, one finds that
\begin{equation}
\label{eqpotmesfinitenu}
\int f(y)  U^{(\nu)}(dy)=\frac{1}{\|\nu\|} f(0)
+ \sum_{k=1}^{\infty} \frac{1}{\|\nu\|^k} \int \cdots \int f(y_1+\cdots +y_k)\nu (dy_1) \cdots \nu (y_k).
\end{equation}

Consequently, by the comparison principle,
the potential measure  $U^{(\nu)}$ has an atom at zero if and only if $\nu$ is finite, in which case this atom is $\de_0/\|\nu\|$.

In the terminology of differential equations, $\tilde G_{(\nu)}(t,.)$ (resp. $G_{(\nu)}(t,.)$) is the Green function of the Cauchy
 problem for the operator $L_{\nu}$ (resp. for the operator $L'_{\nu}$), the measure $U^{(\nu)}(dy)$ on $\{y\ge 0\}$
 is the fundamental solution of the operator $-L'_{\nu}$, and the measure $U^{(\nu)}(-dy)$ on $\{y\le 0\}$ is the fundamental
  solution of the operator $-L_{\nu}$.

In the terminology of semigroups the operator $g\to \int g(x+y)U^{(\nu)}(dy)$ with the kernel $U^{(\nu)}(dy)$ is the potential operator
for the semigroup $T_t$, and the convolution operator $g\to \int g(x-y)U^{(\nu)}(dy)$
is the potential operator for the semigroup $T_t'$.
The $\la$-{\it potential measure}  of the convolution semigroup
$\{ G_{(\nu)}(t,.)\}$ represents the integral kernel of the resolvent operator $R'_{\la}$
of the semigroup $T_t'$ generated by $L'_{\nu}$. On the other hand, $R_{\la}'=(\la-L_{\nu}')^{-1}$, so that
its integral kernel $U^{(\nu)}_{\la}(dy)$ is the fundamental solution of the operator $\la-L_{\nu}'$.

\begin{prop}
\label{propLevyFellersubpotenfund}
Let the measure $\nu$  on $\{y:y>0\}$ satisfy \eqref{eqalevysubo}.

(i) For any $\la >0$, the $\la$-potential measure $U^{(\nu)}_{\la}$ represents the unique fundamental solution
of the operator $\la-L_{\nu}'$.

(ii) If the support of $\nu$ is not contained in a lattice $\{\al n, n\in \Z\}$, with some $\al>0$,
the measure $U^{(\nu)}(dy)$ represents the unique fundamental solution to the operator $-L'_{\nu}$, up to an additive constant.

(iii) Let $\{\al n, n\in \Z\}$ be the minimal lattice containing the support of $\nu$, so that for any $k\in \Z$, $k>1$, there exists
$n\in \Z$ such that $\al n$ belongs to the support of $\nu$ and $n/k \notin \Z$. Then any two fundamental solutions
to the operator $-L'_{\nu}$ differ by a linear combination of the type
\begin{equation}
\label{eq1propLevyFellersubpotenfund}
G=\sum_{n \in \Z} a_n \exp\{2\pi n ix/\al\}
\end{equation}
with some numbers $a_n$. In particular, $U^{(\nu)}(dy)$ is again the unique fundamental solution vanishing on the negative half-line.

\end{prop}

\begin{proof}
(i) For any two fundamental solutions $U_1, U_2$ of $\la-L'_{\nu}$, it follows by the Fourier transform that
$(\psi_{\nu}(-p)-\la) (FG)(p)=0$ for $G=U_1-U_2$. Since $Re \, (\psi_{\nu}(-p)-\la) \le -\la < 0$ for all $p$, $FG(p)=0$ and hence $G=0$.

(ii) For any two fundamental solutions $U_1, U_2$ of $-L'_{\nu}$, it follows by the Fourier transform that
$\psi_{\nu}(-p) (FG)(p)=0$ for $G=U_1-U_2$.  Since the support of $\nu$ is not contained in a lattice, $\psi_{\nu}(-p)<0$ everywhere except at $p=0$,
because $\cos (py)-1<0$ everywhere except when $y=2\pi n/p$ with some $n\in Z$.
Hence $FG$ has support at zero. Consequently, $FG$ is a finite linear combination of
the derivatives $\de^{(j)}$ of the $\de$-function. But the derivative of $\psi_{\nu}(-p)$ at zero does not vanish, it either equals
$-i\int y \nu(dy)$, if this integral is finite, or is not finite at all, if otherwise. In both cases $FG$ can not have
other terms in the sum as $\de$-function itself. Hence $G$ is a constant, as claimed.

(iii) Under assumption of (ii), $\nu(dy)=\sum_{n>0} b_n \de_{\al n} (y)$ with some non-negative numbers $b_n$ such that
for any $k\in \Z$, $k>1$, there exists $n\in \Z$ such that $b_n \neq 0$ and $n/k \notin \Z$. Hence
\[
\psi_{\nu}(p)=  \sum_{n=1}^{\infty} b_n (e^{ip \al n}-1).
\]
Thus $\psi_{\nu}(p)=0$ for $p_m=2\pi m/\al$, $m\in \Z$. Moreover, if $p$ is not of this form, then $\psi_{\nu}(p)\neq 0$.
In fact, assuming otherwise, that is, $\psi_{\nu}(p)=0$ for some $p\neq p_m$. Then $p=(2\pi/\al)(m/k)$ with some rational number $m/k$.
Let us choose it so that the fraction $m/k$ is irreducible. Then $k>1$, since $p\neq p_m$.
Let us choose $n\in \Z$ such that $b_n \neq 0$ and $n/k \notin \Z$.
Then $p \al n=2\pi l$ with some integer $l$, and thus $m/k=l/n$. Since $m/k$ is irreducible, $n/k$ is a integer, leading
 to the contradiction. Consequently, if $\psi_{\nu}(-p) (FG)(p)=0$, then the support of $FG$ is on the lattice $\{p_m\}$.
 As in (i) we derive that the derivatives of $\de$-functions cannot enter in the formula for $FG$. Hence
 $FG(p)= \sum_{m \in \Z} a_m \de_{p_m}(p)$, implying \eqref{eq1propLevyFellersubpotenfund}. The final statement follows,
 because a linear combinations of exponents cannot vanish on the negative half-line.
\end{proof}

We are mostly interested in the Banach-valued extensions of the convolution semigroups.
The following result is obtained analogously to the real case and thus its proof is omitted.

\begin{prop}
\label{propLevyFellersubBval}
(i) Let $B$ be a Banach space. Under condition \eqref{eqalevysubo}, the operators $T_t$ and $T'_t$ given
by \eqref{eq1propLevyFellersub} and \eqref{eqlevysubo4} extend to  $C(\R,B)$ (by the same formula) and
represent strongly continuous semigroups in each of the spaces  $C^k_{\infty}(\R,B)$, $k=0,1, \cdots $ and
$C_{uc}(\R,B)$ (bounded uniformly continuous).

(ii) If $T_t^{\ep}$ and $(T_t^{\ep})'$ denote the semigroups generated by the finite approximations $\nu_{\ep}(dy)=\1_{|y|\ge \ep}\nu(dy)$
of $\nu$, then $T_t^{\ep} \to T_t$ and $(T_t^{\ep})' \to T'_t$ strongly,  as $\ep \to 0$, in each of the spaces $C^k_{\infty}(\R,B)$.

(iii) The space $C^1_{\infty}(\R,B)$ represents an invariant core for both $T_t$ and $T'_t$ in $C_{\infty}(\R,B)$.

(iv) The resolvent operators $R'_{\la}$, given by the formula
\[
 R'_{\la}f(x)=\int_0^{\infty} f(x-y) U_{\la}^{(\nu)} (dy),
 \]
 also extends to the bounded operators in $C_{\infty}(\R,B)$, so that $R'_{\la}(\la-L'_{\nu})f=f$ for any $f\in C_{\infty}(\R,B)$.
\end{prop}

As a consequence, let us prove the following.

\begin{prop}
\label{thLevyFellersubBval}
(i) Let $A$ be an operator in $B$ that generates a strongly continuous semigroup $e^{tA}$ in $B$ with an invariant core $D$.
Let $D$ itself be a Banach space under some norm $\|.\|_D\ge \|.\|_B$
such that $e^{tA}$ is also strongly continuous in $D$ and $A$ is a bounded operator $D\to B$. Then $e^{tA}$ extends to the
strongly continuous semigroup in $C_{\infty}(\R,B)$ with the invariant core $C_{\infty}(\R,D)$.

(ii) The semigroup $e^{tA}$ and the semigroups $T'_t$ generated by $L'_{\nu}$ (as constructed in Proposition \ref{propLevyFellersubBval})
are commuting semigroups in $C_{\infty}(\R,B)$ with $C^1_{\infty}(\R,D)$ being their common invariant core. Moreover, the operator $L'_{\nu}+A$
generates a strongly continuous semigroup $T'_te^{tA}$ in the following space: (a) $C_{\infty}(\R,B)$ with the invariant core $C^1_{\infty}(\R,D)$;
(b) $C_{uc}(\R,B)$ with the invariant core $C^1_{uc}(\R,D)$ of functions from $C^1(\R,D)$, which are uniformly continuous together
 with their derivatives; (c) $C_{uc}((-\infty,b],B)$ for any $b$, with the invariant core $C^1_{uc}((-\infty,b],D)$.
\end{prop}

\begin{proof}
(i) The operators $e^{tA}$ act in $C_{\infty}(\R,B)$ point-wise: $(e^{tA}f)(x)=e^{tA}(f(x))$.
These operators form a bounded semigroup in $C_{\infty}(\R,B)$, because $e^{tA}$ form a bounded semigroup in $B$.
To see strong continuity we note that the point-wise convergence, $e^{tA}(f(x))\to f(x)$, as $t\to 0$
for any $x$, follows from the strong continuity of $e^{tA}$ in $B$. The uniform in $x$ convergence follows then
by the uniform (in $x$) equicontinuity (in parameter $t$) of the family $e^{tA}(f(x))$. Applying the same
result for $D$, we conclude that the operators $e^{tA}$ represent  a strongly continuous semigroup also in $C_{\infty}(\R,D)$.
Finally, for any $f\in C_{\infty}(\R,B)$,
\[
\frac{1}{t}(e^{tA}f(x)-f(x))=\frac{1}{t}\int_0^t Ae^{sA}f(x)\, ds \to Af(x), \quad t\to 0,
\]
uniformly in $x$.

 (ii) The commutativity of $e^{tA}$ and $T'_t$ can be best proved by starting from their approximations
 with a bounded generator (say, the Yosida approximation for $A$ and $(T_t^{\ep})'$ for $T_t$) and then passing to the limit
in the commutation relation. From the commutativity of $T'_t$ and $e^{tA}$, it follow that the operators
$T'_te^{tA}$ form a strongly continuous semigroup in $C_{\infty}(\R,B)$.
Since both $T'_t$ and $e^{tA}$ have common core $C^1_{\infty}(\R,D)$, it follows that
$C^1_{\infty}(\R,D)$ is also a core for $T'_te^{tA}$. Similarly one deals with other spaces mentioned.
\end{proof}

\section{Preliminaries: perturbation theory and its path integral representation}

Let us recall the basic formula of the perturbation theory for linear operators.
Namely, if an operator $A$ with domain
$D_A$ generates a strongly continuous semigroup $e^{tA}$ on a Banach
space $B$ and $L$ is a bounded operator in $B$, then $A+L$ with
the same domain $D_A$ also generates a strongly continuous semigroup
$\Phi_t$ in $B$ given by the series
 \begin{equation}
\label{eqperturbseries}
 \Phi_t =T_t+\sum_{m=1}^{\infty} \int_{0\le s_1\le \cdots \le s_m\le t}
 T_{s_1}LT_{s_2-s_1} \cdots LT_{t-s_m}\, ds_1 \cdots ds_m
 \end{equation}
 converging in the operator norm.

The path space we shall work here will be the space of piecewise constant paths. Namely, a sample path $Z$ in $\R^d$
on the time interval $[0,t]$ and starting at a point $y$ is defined by a finite number, say $n$, of jump-times
$0<s_1 < ...<s_n < t$, and by jumps-sizes $z_1,...,z_n$ (each $z_j\in \R^d\setminus\{0\}$) at these times:
\begin{equation}
\label{eqgenpathforjumpproc}
Z_x(s)=x-Z^{s_1...s_n}_{z_1...z_n}(s), \quad
Z^{s_1...s_n}_{z_1...z_n}(s)=\left\{
 \begin{aligned}
 &  Z^0=0, \quad 0\le s<s_1, \\
 & Z^1=z_1, \quad s_1 \le s <s_2, \\
 & \cdots \\
 & Z^n=z_1+\cdots + z_n, \quad s_n \le s \le t.
 \end{aligned}
  \right.
 \end{equation}
Let $PC_x (s,t)$ (abbreviated to $PC_x(t)$, if $s=0$) denote the set
of all such right-continuous and piecewise-constant paths $[s,t] \mapsto \R^d$ starting from the point $x$,
 and let $PC_x^n(s,t)$ denote the subset of paths with exactly $n$ discontinuities.

Topologically, $PC_x^0$ is a point and $PC_x^n=Sim^n_t \times (\R^d)^n$, $n=1,2,...$, where
\begin{equation}
\label{eqdefsimplex}
Sim_t^n=\{ s_1,...,s_n: 0< s_1 <s_2<...<s_n < t \}
\end{equation}
denotes the standard $n$-dimensional simplex.

To each $\sigma$-finite measure $M$ on $\R^d$, there corresponds a $\sigma$-finite
measure $M^{PC}=M^{PC}(t,x)$ on $PC_x(t)$, which is defined as the sum of measures
$M_n^{PC}$, $n=0,1,...$, where each $M_n^{PC}$ is the product-measure on $PC_x^n(t)$
of the Lebesgue measure on $Sim_t^n$ and of $n$ copies of the measure $M$ on $\R^d$.
Thus if $Z$ is parametrized as in \eqref{eqgenpathforjumpproc}, then
\[
M_n^{PC} (dZ(.))=ds_1....ds_n M(dz_1)...M(dz_n),
\]
and for any measurable functional $F(Z_x(.))=\{F_n(x-Z^0,x-Z^1, \cdots , x-Z^n)\}$ on $PC_x(t)$,
given by the collection of functions $F_n$ on $\R^{dn}$, $n=0,1, \cdots$,
\[
\int_{PC_x(t)} F(Z_x(.)) M^{PC}(dZ(.))
=F(x)+\sum_{n=1}^{\infty}\int_{PC_x^n(t)}  F(Z_x(.)) M_n^{PC}(dZ(.))
\]
\begin{equation}
\label{eqdefpathintpiecewiseconst}
=\sum_{n=0}^{\infty} \int_{Sim_t^n}ds_1...ds_n \int_{\R^d}...\int_{\R^d}M(dz_1)...
M(dz_n) \, F_n(x-Z^0,x-Z^1, \cdots , x-Z^n).
\end{equation}

If the measure $M$ on $\R^d$ is finite, then so is the measure $M^{PC}=M^{PC}(t,x)$ on $PC_x(t)$ with
\[
\|M^{PC}\|=1+\sum_{n=1}^{\infty} \int_{Sim_t^n}ds_1...ds_n \int_{\R^{dn}} M(dz_1)...M(dz_n)
=e^{t\|M\|}.
\]
Hence, using the probabilistic notation $\E$ (the expectation) for the integral over the normalized
(probability) measure $\tilde M^{PC}=e^{-t\|M\|} M^{PC}$ on the path-space $PC_x(t)$,
we can write \eqref{eqdefpathintpiecewiseconst} as
\begin{equation}
\label{eqdefpathintpiecewiseconst1}
\int_{PC_x(t)} F(Z_x(.)) M^{PC}(dZ(.))=e^{t\|M\|} \int_{PC_x(t)} F(Z_x(.)) \tilde M^{PC}(dZ(.))
=e^{t\|M\|} \E \, F(Z_x(.)).
\end{equation}

Let us look now at perturbation series  \eqref{eqperturbseries} assuming that $A$ is the operator of multiplication
by the function $A(y)$ in $\R^d$ and $L$ is a bounded operator in $C(\R^d)$, that for simplicity we take to be spatially homogeneous, that is
$Lf(x)=\int f(x-y)\nu(dy)$ with a measure $\nu$ on $\R^d$ (possibly unbounded and complex-valued).

Then series  \eqref{eqperturbseries} rewrites as
\[
 \Phi_t Y(x)=e^{tA(x)}Y(x)+\sum_{m=1}^{\infty} \int_{0\le s_1\le \cdots \le s_m\le t}Y(x-z_1-\cdots - z_m) ds_1 \cdots ds_m \nu(dz_1) \cdots \nu (dz_m)
 \]
 \begin{equation}
\label{eqperturbseriespathint}
\times \exp\{ s_1 A(x)+(s_2-s_1)A(x-z_1)+\cdots +(t-s_m) A(x-z_1-\cdots - z_m) \}.
 \end{equation}
The latter exponential term can be also written as
\[
\exp\{\int_0^t A(Z_x(s)) \, ds \}.
\]
Comparing with \eqref{eqdefpathintpiecewiseconst}, we derive the following result from \cite{Ko02} (see more detail in Chapter 9 of  \cite{Ko11}).

\begin{theorem}
\label{thperturbationtheorypthint}
Let $A$ be the operator of multiplication
by the function $A(y)$ in $\R^d$ and $Lf(x)=\int f(x-y)\nu(dy)$ with a measure $\nu$ on $\R^d$. Then the convergent series
\eqref{eqperturbseries} or \eqref{eqperturbseriespathint} expressing the resolving operator for the Cauchy problem of equation $\dot f=(A+L)f$ can be represented as
the path integral of type \eqref{eqdefpathintpiecewiseconst} with
\[
F(Z_x(.))= \exp\{\int_0^t A(Z_x(s)) \, ds \} Y(Z_x(t))
\]
and $\nu$ instead of $M$.
 \end{theorem}

\section{Generalized fractional integration and differentiation}
\label{secgenfrac}

The fractional derivative $d^{\be}f/dx^{\be}$, $\be\in (0,1)$, was suggested as a substitute to the usual derivative $df/dx$,
which can model some kind of memory by taking into account the past values of $f$. An obvious extension widely used
in the literature represent various mixtures of such derivatives, both discrete and continuous,
\begin{equation}
\label{eqdefmixedfracder}
\sum_{j=1}^N a_j\frac{d^{\be_j}f}{dx^{\be_j}}, \quad  \int_0^1 \frac{d^{\be}f}{dx^{\be}} \mu (d\be).
\end{equation}

To take this idea further, one can observe that $d^{\be}f/dx^{\be}$ represents a weighted sum of the increments of $f$,
$f(x-y)-f(x)$, from various past values of $f$ to the 'present value' at $x$. From this point of view, the natural
class of {\it generalized mixed fractional derivative} represent the {\it causal integral operators}
\begin{equation}
\label{eqdefmixedfracder}
L_{\nu}f(x)=\int_0^{\infty} (f(x-y)-f(x))\nu (dy),
\end{equation}
with some positive measure $\nu $ on $\{y:y> 0\}$ satisfying the {\it one-sided L\'evy condition}:
 \begin{equation}
\label{eq0diftointfraceqgennu}
\int_0^{\infty} \min(1,y)\nu(y) dy<\infty,
\end{equation}
which ensures that $L_{\nu}$ is well-defined at least on the set of bounded
infinitely smooth functions on $\{y:y\ge 0\}$. The dual operators to $L_{\nu}$
 are given by the {\it anticipating integral operators}
(weighted sums of the increments from the 'present' to any point 'in future'):
\begin{equation}
\label{eqdefmixedfracderdual}
L'_{\nu}f(x)=\int_0^{\infty} (f(x+y)-f(x))\nu (dy).
\end{equation}

Of course, one can weight differently the points in past or future depending on the present position,
and one can also add a local part to complete the picture, leading to the operators
\begin{equation}
\label{eqdefmixedfracdervarl}
L^l_{\nu,b}f(x)=\int_0^{\infty} (f(x-y)-f(x))\nu (x,dy)+b(x)\frac{df}{dx},
\end{equation}
with a non-positive drift $b(x)$ and transition kernel $\nu(x,.)$ such that $\int \min(1,y)\nu(x,dy)<\infty$,
which capture in full the idea of 'weighting the past'
and which can be called the {\it one-sided}, namely {\it left-sided} or {\it causal},
{\it operators of order at most one}.\index{operator of order at most one!left-sided}
Symmetrically, one can define the {\it right-sided} or {\it anticipating}
{\it operators of order at most one}\index{operator of order at most one!right-sided} as
\begin{equation}
\label{eqdefmixedfracdervarr}
L^r_{\nu,b}f(x)=\int_0^{\infty} (f(x+y)-f(x))\nu (x,dy)-b(x)\frac{df}{dx}.
\end{equation}

General {\it operators of order at most one},\index{operator of order at most one}
which represent linear combinations of one-sided operators, and their semigroups were systematically studied in
\cite{Ko10}, \cite{Ko11}. The theory of the corresponding fractional differential equations was built in
\cite{HerHer16b} and \cite{HerHer16}.

For simplicity, let us stick here to the general mixed derivatives
\eqref{eqdefmixedfracderdual} and \eqref{eqdefmixedfracder} and use the notations
\begin{equation}
\label{eqdefmixedfracder}
\begin{aligned}
& D^{(\nu)}_+=-L'_{\nu}f(x)=-\int_0^{\infty} (f(x-y)-f(x))\nu (dy),
\\
& D^{(\nu)}_-=-L_{\nu}f(x)=-\int_0^{\infty} (f(x+y)-f(x))\nu (dy).
\end{aligned}
\end{equation}
With some abuse of notations, if $\nu$ has a density, we shall denote this density again by $\nu$.

The sign $-$ is introduced to comply with the standard notation of the fractional derivatives,
so that, for instance,
\[
\frac{d^{\be}}{dx^{\be}}f(x)=D^{\be}_{-\infty+}=D^{(\nu)}_+
\]
with $\nu(y) = -1/[\Ga (-\be)y^{1+\be}]$, because
\[
\frac{d^{\be}}{dx^{\be}}f(x)=D^{\be}_{-\infty+}f(x)
= \frac{1}{\Ga (-\be)} \int_0^{\infty}\frac{f(x-y)-f(x)}{y^{1+\be}}dy
\]
and $\Ga(-\be)<0$.

The symbols of $\Psi$DOs  $D^{(\nu)}_+$ and  $D^{(\nu)}_-$ are $-\psi_{\nu}(-p)$ and
$-\psi_{\nu}(p)$, where
\[
\psi_{\nu}(p)=\int (e^{ipy}-1) \nu(dy)
\]
is the symbol of the operator $L_{\nu}$.

If $\nu$ is finite, then the operators $D^{(\nu)}_+$  are bounded, which is not the case for the derivatives.
Thus the proper extensions of the derivatives represent only the operator $D^{(\nu)}_+$ arising from infinite measures $\nu$
satisfying \eqref{eq0diftointfraceqgennu}. The operators arising from finite $\nu$ can be better considered as analogs
of the finite differences approximating the derivatives).

The operators $D^{(\nu)}_{\pm}$ represent the extensions of the fractional derivatives $D^{\be}_{-\infty+}$
and $D^{\be}_{\infty-}$, often referred to as the derivatives in the generator form.
Looking for the corresponding extensions of the operators $D^{\be}_{a\pm}$ and $D^{\be}_{a\pm*}$
with a finite $a$ we note that $D^{\be}_{a+*}$ (resp. $D^{\be}_{a-*}$) is obtained from
$D^{\be}_{-\infty+}$ (resp. $D^{\be}_{\infty-}$) by the restriction of its action on the subspace
$C^1([a,\infty))$ (resp. $C^1((-\infty,a])$. Therefore, the analogs of the Caputo derivatives should be defined as
\begin{equation}
\label{eqdefmixedfracderCap}
\begin{aligned}
& D^{(\nu)}_{a+*}=-\int_0^{x-a} (f(x-y)-f(x))\nu (dy) -\int_{x-a}^{\infty} (f(a)-f(x))\nu (dy),
\\
& D^{(\nu)}_{a-*}=-\int_0^{a-x} (f(x+y)-f(x))\nu (dy) -\int_{x-a}^{\infty} (f(a)-f(x))\nu (dy).
\end{aligned}
\end{equation}

Let us denote by $C^k_{kill(a)}([a,\infty))$ and $C^k_{kill(a)}((-\infty,a])$ the subspaces of
$C^k([a,\infty))$ and $C^k((-\infty,a])$ respectively consisting of functions vanishing to the right or to the left of $a$.

On the other hand, the operators $D^{\be}_{a+}$ or $D^{\be}_{a-}$,
the analogs of the Riemann-Liouville derivatives, are obtained by further restricting the actions of
$D^{\be}_{-\infty+}$ and $D^{\be}_{\infty-}$  to the spaces $C^1_{kill(a)}([a,\infty))$ and $C^1_{kill(a)}((-\infty,a])$:
\begin{equation}
\label{eqdefmixedfracderRL}
\begin{aligned}
& D^{(\nu)}_{a+}=-\int_0^{x-a} (f(x-y)-f(x))\nu (y) dy+\int_{x-a}^{\infty} f(x)\nu (y) dy,
\\
& D^{(\nu)}_{a-}=-\int_0^{a-x} (f(x+y)-f(x))\nu (y) dy+\int_{x-a}^{\infty} f(x))\nu (y) dy.
\end{aligned}
\end{equation}

To see what should be the proper analog of the fractional integral, notice that, as is known (see e.g. \cite{Gel2}
or \cite{Ko15a}), the fundamental solution (vanishing on the negative half-line)
to the fractional derivative $d^{\be}/dx^{\be}$ is $U^{\be}(x)=x^{\be-1}_+/\Ga (\be)$, so that the usual fractional integral
\begin{equation}
\label{eqdeffracintaspotbe}
I_a^{\be}f(x)=\frac{1}{\Ga(\be)}\int_a^x (x-y)^{\be-1}f(y) \, dy
=\frac{1}{\Ga(\be)} \int_0^{x-a} z^{\be-1}f(x-z) \, dz
\end{equation}
is nothing else but the
potential operator of the semigroup generated by $-d^{\be}/dx^{\be}$,
or, in other words, the integral operator with the kernel being the fundamental solution
of $-d^{\be}/dx^{\be}$ (or, yet in other words, the convolution with this fundamental solution),
restricted to the space $C_{kill(a)}([a,\infty))$.

By Proposition \ref{propLevyFellersubpotenfund},
the potential measure $U^{(\nu)}(dy)$ represents the unique fundamental solution to the operator $L'_{\nu}$,
vanishing on the negative half-line. Hence the analog of the fractional integral $I^{\be}_a$ for such $\nu$
should be the potential operator of the semigroup $T_t'$ generated by  $L'_{\nu}$, that is, the convolution
with $U^{(\nu)}(dy)$ restricted to the space $C_{kill(a)}([a,\infty))$:
\begin{equation}
\label{eqdefgenfracint}
I^{(\nu)}_af(x) =\int_0^{x-a}f(x-z) U^{(\nu)}(dz).
\end{equation}

The following result corroborates this identification.

\begin{prop}
\label{propLevyFellersubpotenfund1}
(i) Let the measure $\nu$  on $\{y:y>0\}$ satisfy \eqref{eq0diftointfraceqgennu}.
For any generalized function $g\in D'(\R)$ supported on the half-line $[a, \infty)$ with any $a\in \R$,
and any $\la \ge 0$, the convolution $U_{\la}^{(\nu)} \star g$ with  the $\la$-potential measure
is a well-defined element of $D'(\R)$, which is also supported on
$[a, \infty)$. This convolution represents the unique solution (in the sense of generalized function)
of the equation $(\la-L'_{\nu})f=g$, or equivalently
\[
D^{(\nu)}_+ f=-\la f +g,
\]
supported on $[a, \infty)$.

(ii) If $\la >0$ and $g\in C_{\infty}(\R)$ and is supported on the half-line $[a, \infty)$,
that is $g\in C_{kill(a)}([a,\infty))$, then
\[
(U_{\la}^{(\nu)} \star g)(x)=R_{\la}'g(x)
=\int_{-\infty}^{\infty} g(x-y) U^{(\nu)}_{\la}(dy)
\]
\begin{equation}
\label{eq1propLevyFellersubpotenfund1}
=\int_0^{x-a} g(x-y) \int_0^{\infty} e^{-\la t}  G_{(\nu)}(t,dy) \, dt
\end{equation}
belongs to the domain of the operator $L'_{\nu}$ and thus represents the classical solution to the
equation $(\la-L'_{\nu})f=g$, or equivalently
\begin{equation}
\label{eq2propLevyFellersubpotenfund1}
D^{(\nu)}_+ f=D^{(\nu)}_{a+} f=D^{(\nu)}_{a+*} f=-\la f +g,
\end{equation}

(iii) If $\la=0$, the potential $U^{(\nu)}$ defines an unbounded operator in $C_{\infty}(\R)$.
 However, if reduced to the space  $C_{kill(a)}([a,b])$
of continuous functions on $[a,b]$ vanishing at $a$ (this space is invariant under $T_t'$ and hence under all $R'_{\la}$),
the potential operator $R'_0$ with the kernel $U^{(\nu)}$ becomes bounded and hence
\begin{equation}
\label{eq3propLevyFellersubpotenfund1}
(U^{(\nu)} \star g)(x)=R'_0g(x)=I^{(\nu)}_af(x)
\end{equation}
belongs to the domain of $L'_{\nu}$ and thus represents the classical solution to the equation
\begin{equation}
\label{eq4propLevyFellersubpotenfund1}
-L'_{\nu}f=D^{(\nu)}_+f=D^{(\nu)}_{a+} f=D^{(\nu)}_{a+*} f=g
\end{equation}
on $C_{kill(a)}([a,b])$.
\end{prop}

\begin{proof}
(i) The convolution $U_{\la}^{(\nu)} \star g$ is well-defined, because of the assumptions on the support of $U_{\la}^{(\nu)}$ and $g$,
 and solves the equation $(\la-L'_{\nu})f=g$, because $U_{\la}^{(\nu)}$ is the fundamental solution.
Uniqueness follows as in Proposition \ref{propLevyFellersubpotenfund}.

(ii)  Since $L'_{\nu}$ generates a semigroup $T_t$ from \eqref{eqlevysubo4}, which preserves the spaces
 $C([a,\infty))$ and $C_{kill(a)}([a,\infty))$, these spaces are also invariant
 under the the resolvent $R_{\la}'=(\la-L_{\nu}')^{-1}$. The image of the resolvent always
coincides with the domain of the generator.
 Hence $R_{\la}'g$ belongs to the intersection of $C_{kill(a)}([a,\infty))$ and the domain of $D^{(\nu)}_+$.

(iii) The potential operator
\[
R'_0g(x)=(U^{(\nu)} \star g)(x)= \int_0^{x-a} g(x-y)U^{(\nu)}(dy)
\]
is bounded on $C_{kill(a)}([a,b])$ because $U^{(\nu)}$ is bounded on compact segments.
\end{proof}

\begin{remark}
The classical interpretation of the solution
$R_{\la}'g(x)$ is subtle for $g\in C([a,\infty))$ not vanishing at $a$.
If $\nu$ is not finite, then $R'_0g$ is continuous at zero even if $g\in C([a,b])$ does not vanish at zero.
Still it does not belong to the domain of $L_{\nu}'$. However, it can be shown to belong to the domain locally, outside
the boundary point $a$. The requirement for the solution to belong to the domain outside a boundary point is common
for the classical problems of PDEs. The following assertion illustrates this point concretely.
\end{remark}

\begin{prop}
\label{propLevyFellersubpotenfund2}
Under the assumptions of Proposition \ref{propLevyFellersubpotenfund1} let the potential measure $U^{(\nu)}(dy)$ have a continuous density,
$U^{(\nu)}(y)$, with respect to Lebesgue measure. Let $g\in C^1[a,b]$. Then the function $f=R'_0g(x)$ belongs to
$C_{kill(a)}([a,b])$ and is continuously differentiable in $(a,b]$. Consequently,
it satisfies the equation $D^{(\nu)}_{a+*} f=g$ locally, at all points from $(a,b]$.
\end{prop}

\begin{proof}
From the formula for $R'_0g(x)$ it follows that
\[
(d/dx)R'_0g(x)=\int_0^{x-a}\frac{d}{dx} g(x-y) U^{(\nu)}(y) \, dy +g(a)U^{(\nu)}(x-a),
\]
which is well-defined and continuous for $x>a$.
\end{proof}

As was mentioned, the image of the resolvent coincides with the domain of the generator implying that function \eqref{eq1propLevyFellersubpotenfund1}
belongs to the domain of $L'_{\nu}$, restricted to $C_{kill(a)}([a,\infty))$, whenever $g\in C_{kill(a)}([a,\infty))$.
For other $g$ our generalized solution was defined in the sense of generalized function. As usual,
one can also introduce the notions of generalized solution by approximation. Namely,  for a measurable bounded function $g(x)$
on $[a, \infty)$, a continuous curve $f(x)$, $t\ge a$, is the {\it generalized solution by approximation}
\index{generalized solution to the Cauchy problem!by approximation} to the problem $D^{(\nu)}_+ f=-\la f +g$ on $C([a, b])$,
if there exists a sequence of the curves $g^n(.) \in C_{kill(a)}([a,b])$ such that $g^n \to g$ a.s., as $n \to \infty$,
and the corresponding classical (i.e. belonging to the domain) solutions $f^n(x)$, given by  \eqref{eq1propLevyFellersubpotenfund1}
with $g^n(x)$ instead of $g(x)$, converge point-wise to $f(t)$, as $n\to \infty$.

The following assertion is a consequence of Proposition \ref{propLevyFellersubpotenfund1}.
\begin{prop}
\label{linderfenRLgensol}
For any measurable bounded function $b(x)$ on $[a, \infty)$, formula \eqref{eq1propLevyFellersubpotenfund1}
(resp. \eqref{eq3propLevyFellersubpotenfund1}) supplies the unique  generalized solution by approximation to
problem \eqref{eq2propLevyFellersubpotenfund1} (resp. \eqref{eq4propLevyFellersubpotenfund1}) on $[a,b]$ for any $b>a$.
\end{prop}

\section{Time-homogeneous case: arbitrary $\nu$}
\label{secgenfraclin}

Extending Proposition \ref{propLevyFellersubpotenfund1}, let us analyse the linear equations with a non-vanishing boundary value at $a$.

 \begin{prop}
 \label{diftointfraceqgennu}
 Let a non-negative measure $\nu $ on $\{y:y> 0\}$ satisfy \eqref{eq0diftointfraceqgennu}.

(i) If $g$ is a generalized function (from $S'(\R)$ or $D'(\R)$) vanishing to the left of $a$, then
 \begin{equation}
\label{eq1diftointfraceqgennu}
f(x)=Y +I^{(\nu)}_ag(x)= Y+ \int_0^{x-a} g(x-y) U^{(\nu)}(dy)=Y+(g\star  U^{(\nu)})(x)
\end{equation}
is the unique solution (from $S'(\R)$ or $D'(\R)$ respectively) to the equation
 \begin{equation}
\label{eq2diftointfraceqgennu}
g=D^{(\nu)}_{a+*} f=D^{(\nu)}_+f
\end{equation}
that equals to the constant $Y$ to the left of $a$.

 (ii) If $g\in C_{kill(a)}([a,b])$, then $f$ from \eqref{eq1diftointfraceqgennu} belongs to the domain of the generator
 of the semigroup $T_t$ defined either on the space $C_{uc}((-\infty, b])$ of uniformly continuous functions on
 $(-\infty, b]$ or on its subspace $C([a,b])$ of functions that are constants to the left of $a<b$. In this case
 $f$ represents the classical solution of equation \eqref{eq2diftointfraceqgennu} that equals $Y$ to the left of $a$.

(iii) If  $g\in C([a,b])$ and $\nu$ is not finite, then $f\in C(-\infty, b]$ for any $b>a$ and thus takes the
initial condition $f(a)=Y$ in the classical sense.
 \end{prop}

\begin{remark}
If $\nu$ is finite and $g(a)\neq 0$, then $f$ has a discontinuity at $a$, as in this case the limit of $f$ from the right at $a$ equals $Y+g(a)\|\nu\|$.
 \end{remark}

\begin{proof} (i) By Propositions \ref{propLevyFellersubpotenfund},
for any $g\in S'(\R)$ supported on $[a,\infty)$, the function $I^{(\nu)}_ag(x)$ is the unique, up to an additive constant,
 solution to the equation $g=D^{(\nu)}_{a+*} f$, in the sense of generalized functions, which is a constant to the left of $a$.
Thus adding $Y$ fixes the initial condition in the unique way.

(ii) As in  Proposition \ref{propLevyFellersubpotenfund1}, this follows from the fact that the image of the potential operator, when it is bounded,
coincides with the domain.

(iii) This holds because $U^{(\nu)}$ has no atoms at the origin.
\end{proof}

\begin{prop}
\label{lindergenCapreal}
Let the measure $\nu$  on $\{y:y>0\}$ satisfy \eqref{eqalevysubo} and let $\la >0$.

(i) For any  $g \in C_{\infty}[a, \infty)$ (considered as the element of $C_{uc}(\R)$ by extending it to the left
 of $a$ by the constant $g(a)$), the function
 \begin{equation}
\label{eq1lindergenCapreal}
f(x)=\int_0^{x-a} g(x-y) U_{\la}^{(\nu)}(dy)=(g\star  U_{\la}^{(\nu)})(x)
\end{equation}
is the unique  solution to the equation
  \begin{equation}
\label{eqlinder1unbgenreal0}
D^{(\nu)}_{a+*}f(x)=-\la f(x)+g(x)
\end{equation}
in the domain of the generator of the semigroup $T_t$ on $C_{uc}(\R)$. This function equals $g(a)/\la$ to the left of $a$.

(ii) For any $g\in S'(\R)$ that is constant to the left of $a$, the generalized function $g\star  U_{\la}^{(\nu)}$
is a well-defined element of $S'(\R)$, and it represents the unique solution to equation \eqref{eqlinder1unbgenreal0}
in the sense of generalized functions.
\end{prop}

\begin{proof}
(i) By \eqref{eq1lindergenCapreal}, $f=R'_{\la}g$ is obtained by applying the resolvent to $g$.
Hence it belongs to the domain of $L_{\nu}$ and solves the equation $(\la -L'_{\nu})f=g$.
(ii) This follows from Proposition \ref{propLevyFellersubpotenfund} (i).
\end{proof}

As above, one can also interpret formula  \eqref{eq1lindergenCapreal} in the sense of generalized solutions by approximation.
However, function \eqref{eq1lindergenCapreal} is not the solution we are mostly interested in, as it prescribes the boundary
value at $a$, rather than solves the boundary value problem.

The most straightforward way to deal properly with the problem
 \begin{equation}
\label{eqlinder1unbgenreal}
D^{(\nu)}_{a+*}f(x)=-\la f(x)+g(x), \quad f(a)=Y, \quad x\ge a,
\end{equation}
is by turning it to the problem with the vanishing boundary value, which is a usual trick in the theory of PDEs.
Namely, introducing the new unknown function $u=f-Y$ we see that $u$ must solve the problem
 \begin{equation}
\label{eqlinder1unbgenrealRLshift}
D^{(\nu)}_{a+}u(x)=-\la u(x)-\la Y +g(x), \quad u(a)=0, \quad x\ge a,
\end{equation}
just with $g-\la Y$ instead of $g$. We can thus {\it define the solution to} \eqref{eqlinder1unbgenreal} to be
the function $f=u+Y$, where $u$ solves \eqref{eqlinder1unbgenrealRLshift}.
Such definition also complies with one of the definition of $D^{(\nu)}_{a+*}$ as given by
$D^{(\nu)}_{a+*}f=D^{(\nu)}_{a+*}(f-f(a))$.

Taking first $g=0$ we find the solution to \eqref{eqlinder1unbgenreal} to be
\[
f(x)=Y+u(x)=Y-\la Y \int_0^{x-a} \int_0^{\infty} e^{-\la t}  G_{(\nu)}(t,dy) \, dt
\]
 \begin{equation}
\label{eqlinder1unbgenrealsol0}
=\la Y \int_0^{\infty} e^{-\la t}  \left(  \int_{x-a}^{\infty} G_{(\nu)}(t,dy)\right) \, dt.
\end{equation}
Integrating by parts we get for $x>a$ an alternative expression:
 \begin{equation}
\label{eqlinder1unbgenrealsol}
f(x)=Y \int_0^{\infty} e^{-\la t}  \frac{\pa}{\pa t}\left(  \int_{x-a}^{\infty} G_{(\nu)}(t,dy)\right) \, dt.
\end{equation}

Restoring $g$ we arrive at the following.

\begin{prop}
\label{lindergenCaprealviashift}
For any $g$ supported on $[a, \infty)$ the unique solution to problem \eqref{eqlinder1unbgenreal}
in the sense defined above is given by the formula
 \[
f(x)=Y \int_0^{\infty} e^{-\la t}  \frac{\pa}{\pa t}\left(  \int_{x-a}^{\infty} G_{(\nu)}(t,dy)\right) \, dt
\]
 \begin{equation}
\label{eqlinder1unbgenrealsol}
+\int_0^{x-a} g(x-y) \int_0^{\infty} e^{-\la t}  G_{(\nu)}(t,dy) \, dt
\end{equation}
This solution can be classified as classical (from the domain of the generator) or generalized
(in the sense of the generalized functions or by approximation) according to Proposition \ref{propLevyFellersubpotenfund1}
applied to problem \eqref{eqlinder1unbgenrealRLshift}.
\end{prop}

As for  $L'_{\nu}=-d^{\be}/dx^{\be}$, the coefficient at $Y$ for $x-a=1$, $E_{\be}(-\la)$, represents
the  Mittag-Leffler function of index $\be$, one can define the analog of the {\it Mittag-Leffler function}
 for arbitrary $\nu$ as
 \[
E_{(\nu)}(-\la)= \int_0^{\infty} e^{-\la t}  \frac{\pa}{\pa t}\left(  \int_1^{\infty} G_{(\nu)}(t,dy)\right) \, dt
\]
 \begin{equation}
\label{eqdefgenML}
=\la \int_0^{\infty} e^{-\la t}  \left(  \int_1^{\infty} G_{(\nu)}(t,dy)\right) \, dt
=1-\la \int_0^{\infty} e^{-\la t}  \left(  \int_0^1 G_{(\nu)}(t,dy)\right) \, dt.
\end{equation}
the function $ \int_{x-a}^{\infty} G_{(\nu)}(t,dy)$ increases with $t$.
Hence its derivative is well-defined as a positive measure (and as a function almost everywhere),
and therefore the function $E_{(\nu)}(-\la)$ is a completely monotone function of $\la$.
This function is well defined and continuous for $Re \, \la \ge 0$, as there it is bounded by $1$:
 \begin{equation}
\label{eqdefgenMLest}
|E_{(\nu)}(-\la)| \le \int_0^{\infty} \frac{\pa}{\pa t}\left(  \int_1^{\infty} G_{(\nu)}(t,dy)\right) \, dt
=\left.\left(  \int_1^{\infty} G_{(\nu)}(t,dy)\right)\right|_0^{\infty}=1.
\end{equation}
Moreover, $E_{(\nu)}(0)=1$.

In fact, one can define the family of these Mittag-Leffler functions depending on the positive parameter $z$ as
  \begin{equation}
\label{eqdefgenMLfam}
E_{(\nu),z}(-\la)= \int_0^{\infty} e^{-\la t}  \frac{\pa}{\pa t}\left(  \int_z^{\infty} G_{(\nu)}(t,dy)\right) \, dt
=1-\la \int_0^{\infty} e^{-\la t}  \left(  \int_0^z G_{(\nu)}(t,dy)\right) \, dt.
\end{equation}
They all are  completely monotone and the solution \eqref{eqlinder1unbgenrealsol0} to problem \eqref{eqlinder1unbgenrealRLshift}
is then expressed as
  \begin{equation}
\label{eqdefgenMLfamsol}
f(x)=Y  E_{(\nu),x-a}(-\la)+\int_0^{x-a} g(x-y) U_{\la}^{(\nu)}(dy),
\end{equation}
where the $\la$-potential measure is expressed in terms of $E_{(\nu),z}$ by the equation
  \begin{equation}
\label{eqdefgenMLfamsol1}
\int_0^z U_{\la}^{(\nu)}(dy) =(1- E_{(\nu),z}(-\la))/\la.
\end{equation}

If the measures $G_{(\nu)}(t,dy)$ have densities with respect to Lebesgue measure, $G_{(\nu)}(t,y)$, then
the $\la$-potential measure also has a density, $U_{\la}^{(\nu)}(y)$, and \eqref{eqdefgenMLfamsol1} rewrites as
  \begin{equation}
\label{eqdefgenMLfamsol2}
U_{\la}^{(\nu)}(y) =- \frac{1}{\la} \frac{\pa E_{(\nu),y}}{\pa y}.
\end{equation}
However, only for the case of the derivative $d^{\be}/dx^{\be}$, due to the particular scaling property of $G_{\be}$,
one has the additional relation $E_{(\nu),z}(-\la)=E_{(\nu)}(-\la z^{\be})$.

In order for $E_{(\nu)}(s)$ to be an entire analytic function, as for the case of usual Mittag-Leffler functions,
some regularity assumptions on $\nu$ are needed, as will be discussed in the next section.

Let us now turn to the extension of the linear equations to the Banach-space-valued setting, that is, to the equations
 \begin{equation}
\label{eqlinder1unbgen}
D^{(\nu)}_{a+*}\mu(x)=A \mu (x)+g(x), \quad \mu(a)=Y, \quad x\ge a.
\end{equation}
If $\mu(a)=Y=0$, this turns to the RL type equation
  \begin{equation}
\label{eqlinder1unbgenRL}
D^{(\nu)}_{a+}\mu(x)=A \mu(x)+g(x), \quad \mu(a)=0, \quad x\ge a.
\end{equation}

As above, we shall define the solution to \eqref{eqlinder1unbgen} as the function $\mu(x)=Y+u(x)$, where $u(x)$ solves the problem
 \begin{equation}
\label{eqlinder1unbgenRLsh}
D^{(\nu)}_{a+}u(x)=A u(x) +AY +g(x), \quad u(a)=0, \quad x\ge a.
\end{equation}

The only new point as compared with real-valued $A$ is the application of Proposition \ref{thLevyFellersubBval}
to build the semigroup $T'_te^{tA}$ and the necessity to work with Banach-space valued generalized functions
if interested in the appropriate interpretation of generalized solutions.
Notice also that the assumption of $e^{tA}$ to be a contraction naturally extends the case $A=-\la$ with $\la>0$,
as $e^{-\la t} \le 1$, and allows one to define the operator-valued generalized Mittag-Leffler functions by the operator-valued integral
  \begin{equation}
\label{eqdefgenMLfamoper}
E_{(\nu),z}(A)= \int_0^{\infty} e^{ tA}  \frac{\pa}{\pa t}\left(  \int_z^{\infty} G_{(\nu)}(t,dy)\right) \, dt
=1+A \int_0^{\infty} e^{ tA}  \left(  \int_0^z G_{(\nu)}(t,dy)\right) \, dt.
\end{equation}

\begin{theorem}
\label{thBvalgenfrac}
(i) Let the measure $\nu$  on $\{y:y>0\}$ satisfy \eqref{eqalevysubo} and let $A$ be the generator of the strongly continuous
semigroup $e^{tA}$ of contractions in the Banach space $B$, with the domain of the generator $D \subset B$.
Then the $\LC(B,B)$-valued potential measure,
 \begin{equation}
\label{eq1thBvalgenfrac}
U^{(\nu)}_{-A}(M)=\int_0^{\infty} e^{tA} G_{(\nu)}(t,M) \, dt,
\end{equation}
of the semigroup $T'_te^{tA}$ on the subspace $C^k_{kill(a)}([a,b],B)$ of $C_{uc}((-\infty,b],B)$
(constructed in Proposition \ref{thLevyFellersubBval}),
is well-defined as a $\si$-finite measure on $\{y:y\ge 0\}$ such that for any $z,\la>0$
\[
 U^{(\nu)}_{-A}([0,z]) \le e^{\la z}/\phi_{\nu}(\la).
\]
Therefore, the potential operator (given by convolution with  $U^{(\nu)}_{-A}$) of the semigroup $T'_te^{tA}$ on
 $C^k_{kill(a)}([a,b],B)$ is bounded for any $b>a$.

(ii) For any $g\in C_{kill(a)}([a,b],B)$, the $B$-valued function
 \begin{equation}
\label{eq2thBvalgenfrac}
f(x)=\int_0^{x-a} U^{(\nu)}_{-A}(dy) g(x-y)=\int_0^{x-a} \int_0^{\infty} e^{tA} G_{(\nu)}(t,dy) \, dt \, g(x-y)
\end{equation}
belongs to the domain of the generator of the semigroup $T'_te^{tA}$ and represents the unique solution to problem \eqref{eqlinder1unbgenRL}
from the domain. For any $g\in C([a,b],B)$, this function represents the unique generalized solution to \eqref{eqlinder1unbgenRL},
both by approximation and in the sense of generalized functions.

(iii) For any $g\in C([a,b],B)$ and $Y\in B$, the function
\begin{equation}
\label{eq3thBvalgenfrac}
f(x)=Y+\int_0^{x-a} U^{(\nu)}_{-A}(dy) (AY +g(x-y))
=E_{(\nu),x-a}(A)Y+\int_0^{x-a} U^{(\nu)}_{-A}(dy) g(x-y)
\end{equation}
represents the unique generalized solution to problem  \eqref{eqlinder1unbgen}.
\end{theorem}

\begin{proof}
(i)  For the measure  $U^{(\nu)}_A$ we obtain the same estimate as for $U^{(\nu)}$, see \eqref{eqestimpotmes},
because $e^{tA}$ are contractions.
(ii) What concerns the solutions in the domain, this is again the consequence of the fact that resolvent
 maps the whole space in the domain of the generator.
Existence and uniqueness of generalized solutions is a consequence of the explicit integral formula.
(iii) This follows from (ii) by the definition of the solution to \eqref{eqlinder1unbgen}.
\end{proof}

\section{Time-homogeneous case: $\nu$ comparable with the stable subordinators}
\label{secgenfracasym}

We have constructed the solutions to the linear problems \eqref{eqlinder1unbgenRL} and \eqref{eqlinder1unbgen} only for the
case of $A$ generating a contraction semigroup (with a direct extension to the case of a uniformly bounded semigroup $e^{tA}$).
This restriction was ultimately linked with formula \eqref{eqdefgenML} for the generalized
Mittag-Leffler function, from which it is not seen directly that it can be extended to negative $\la$.
Here we shall present some additional assumptions on $\nu$ that would ensure that this extension is possible
and thus the results above could be extended to $A$ generating arbitrary strongly continuous semigroups.
These assumptions are of two kinds, via lower bounds for $\nu(dy)$ and via its asymptotics at small $y$.

In the case of bounded operators $A$ in a Banach space $B$ the natural construction of the solutions to the linear problem
$D^{(\nu)}_{a+*} f(x)=Af(x)$ with a given initial condition $f(x)=Y$ is by rewriting it in the integral form
(by Proposition \ref{propLevyFellersubpotenfund2})and then representing it by the geometric series of the operators
$I^{(\nu)}_a$ having the potential measure as the kernel:
  \begin{equation}
\label{eqdefgenMitLeffu}
(1+A(I^{(\nu)}_a \1)(x) +\cdots +A^k [(I^{(\nu)}_a)^k \1](x) +\cdots )Y,
\end{equation}
whenever it converges.
Thus we are looking for the assumptions on $\nu$, which can ensure the convergence and provide reasonable
estimates for the sum.

The following assertion is the consequence of the comparison principle, see \eqref{eqcomppotmes},
and the well known expression for the potential measures of stable subordinators.
 \begin{prop}
 \label{generlinfraccomp1}
 Let $\nu (dy)$ be a measure on $\{y:y> 0\}$
 satisfying \eqref{eq0diftointfraceqgennu} and having the lower bound of the $\be$-fractional type
 \begin{equation}
\label{eq1generlinfraccomp1}
\nu(dy) \ge (-1/\Ga (-\be)) C_{\nu} y^{-1-\be} \, dy
\end{equation}
with some $\be \in (0,1)$ and $C_{\nu}>0$. Then
\begin{equation}
\label{eq2generlinfraccomp1}
 \int_0^x U^{(\nu)}(dy) \le C_{\nu}(I_0^{\be}\1)(x)=C_{\nu}x^{\be}/\Ga(\be)
\end{equation}
for any $x>0$.
 \end{prop}

By \eqref{eq2generlinfraccomp1},
\[
\| 1+\la(I^{(\nu)}_a \1)(x) +\cdots +\la^k [(I^{(\nu)}_a)^k \1](x) +\cdots \|
\]
\begin{equation}
\label{eq3generlinfraccomp1}
\le \| 1+C_{\nu}|\la|(I^{\be}_a \1)(x) +\cdots +(C_{\nu}|\la|)^k [I^{k\be}_a \1](x) +\cdots \|
\le E_{\be}(C_{\nu}|\la|(x-a)^{\be}).
\end{equation}

 \begin{theorem}
 \label{genpropMFreg}
 Under the assumptions of Proposition \ref{generlinfraccomp1}
 the integral \eqref{eqdefgenMLfam} converges for all complex $\la$, so that the function $E_{(\nu),z}(\la)$
  (defined initially  by \eqref{eqdefgenMLfam} for the negative values of parameter)
 represents an entire analytic function of $\la$. Its series expansions is
 \begin{equation}
\label{eq1genpropMFreg}
E_{(\nu),z}(\la)= 1+\la (I^{(\nu)}_0 \1)(z) +\cdots +\la^k [(I^{(\nu)}_0)^k \1](z) +\cdots,
 \end{equation}
 or
  \[
E_{(\nu),z}(\la)= 1+\la (I^{(\nu)}_a \1)(x) +\cdots +\la^k [(I^{(\nu)}_a)^k \1](x) +\cdots,
 \]
 with $x-a=z$. It can be also obtained by expanding the last expression of
 \eqref{eqdefgenMLfam} in power series over $\la$. Series \eqref{eq1genpropMFreg}
 is bounded by  \eqref{eq3generlinfraccomp1}.

 Moreover, the integral expressing the $\la$-potential measure
 \[
 U^{(\nu)}_{\la}([0,z])=\int_0^{\infty} e^{-\la t} G_{(\nu)}(t, [0,z])  \, dt
 \]
  converges for all complex $\la$, so that
  the $\la$-potential measure is also an
 entire analytic function of $\la$ and its series expansions is obtain from that of $E_{(\nu),z}(-\la)$
 via formula \eqref{eqdefgenMLfamsol1}. Finally,
 \begin{equation}
\label{eq2genpropMFreg}
\|U^{(\nu)}_{\la}([0,z])\|\le C_{\nu} \be z^{\be-1} E'_{\be} (|\la| z^{\be}).
 \end{equation}
 \end{theorem}

\begin{proof}
Expanding  the last expression of \eqref{eqdefgenMLfam} in the power series over $\la$,
we see, by the comparison principle, that
 all terms are bounded by the corresponding terms of the series with $G_{\be}(t,dy)$ instead of $G_{(\nu)}(t,dy)$.
 Hence this series is convergent for all $\la$. Since both  the last expression in
 \eqref{eqdefgenMLfam} and series \eqref{eq1genpropMFreg} solve the same linear fractional equation, they coincide.

Again by the comparison principle,
\[
\|U^{(\nu)}_{\la}([0,z])\|\le \int_0^{\infty} e^{|\la| t} G_{(\nu)}(t, [0,z]) \, dt
\le  C_{\nu} \int_0^{\infty} e^{|\la| t} G_{\be}(t, [0,z]) \, dt,
\]
implying \eqref{eq2genpropMFreg}.
\end{proof}

We are ready for the main result of this section that extends Theorem \ref{thBvalgenfrac} to arbitrary
 semigroups $e^{tA}$.  The proof is fully the same as that of Theorem \ref{thBvalgenfrac}
(once the properties of the $\la$-potential measures from Theorem \ref{genpropMFreg} are obtained)
 and is thus omitted.

 \begin{theorem}
 \label{thgenlinfraceqreg}
 Under the assumptions of Proposition \ref{generlinfraccomp1}
let $A$ be the generator of the strongly continuous
semigroup $e^{tA}$ in the Banach space $B$, with the domain of the generator $D \subset B$.
Let the growth type of $e^{tA}$ is $m_0$, so that $\|e^{tA}\|\le Me^{mt}$ with any $m>m_0$ and some $M$.
Then the following holds.

(i) The $\LC(B,B)$-valued potential measure,
 \begin{equation}
\label{eq1thgenlinfraceqreg}
U^{(\nu)}_{-A}(M)=\int_0^{\infty} e^{tA} G_{(\nu)}(t,M) \, dt,
\end{equation}
of the semigroup $T'_te^{tA}$ on the subspace $C^k_{kill(a)}([a,b],B)$ of $C_{uc}((-\infty,b],B)$
(constructed in Theorem \ref{thLevyFellersubBval}), is well-defined as a $\si$-finite measure on $\{y:y\ge 0\}$ and
 \begin{equation}
\label{eq2thgenlinfraceqreg}
 U^{(\nu)}_{-A}([0,z]) \le  C_{\nu} M \be z^{\be-1} E'_{\be} (m z^{\be})
\end{equation}
for any $z>0$.

(ii) The $\LC(B,B)$-valued generalized families of Mittag-Leffler functions
  \begin{equation}
\label{eq3thgenlinfraceqreg}
E_{(\nu),z}(A)= \int_0^{\infty} e^{A t}  \frac{\pa}{\pa t}\left(  \int_z^{\infty} G_{(\nu)}(t,dy)\right) \, dt
=1+A \int_0^{\infty} e^{A t}  \left(  \int_0^z G_{(\nu)}(t,dy)\right) \, dt
\end{equation}
are well-defined and are bounded:
  \begin{equation}
\label{eq4thgenlinfraceqreg}
\|E_{(\nu),z}(A)\|\le  M E_{\be}(C_{\nu} m (x-a)^{\be}).
\end{equation}

(iii) For any $g\in C_{kill(a)}([a,b],B)$, the $B$-valued function \eqref{eq2thBvalgenfrac}
belongs to the domain of the generator of the semigroup $T'_te^{tA}$ and represents the unique solution to problem \eqref{eqlinder1unbgenRL}
from the domain. For any $g\in C([a,b],B)$, this function represents the unique generalized solution to \eqref{eqlinder1unbgenRL},
both by approximation and in the sense of generalized functions.

(iv) For any $g\in C([a,b],B)$ and $Y\in B$, the function \eqref{eq3thBvalgenfrac}
represents the unique generalized solution to problem  \eqref{eqlinder1unbgen}.
 \end{theorem}

\section{Time-nonhomogeneous case: bounded $\nu$}
\label{secgenfractimedep}

Our aim now is to extend the results obtained above for \eqref{eqlinder1unbgen} to the case of the family of operators $A$ depending on $x$, that is,
to the problem
 \begin{equation}
\label{eqlinder1unbgentimedep}
D^{(\nu)}_{a+*}\mu(x)=A(x) \mu (x)+g(x), \quad \mu(a)=Y, \quad x\ge a.
\end{equation}

This development is based on an appropriate extension of Theorem \ref{thLevyFellersubBval}, which we shall carry out in two steps,
first for bounded and then for unbounded measures $\nu$. In any case, the method of three spaces turns out to be convenient.

\begin{theorem}
\label{thLevyFellersubBvaltimedep}
(i) Let $\tilde D\subset D\subset B$ be three Banach spaces with the ordered norms: $\|.\|_{\tilde D} \ge \|.\|_D \ge \|.\|_B$
and $\tilde D$ is dense in both $D$ and $B$ with respect to their topologies (three Banach spaces setting). Let
 $A(x)$, $x\in \R$, be a uniformly bounded family of operators in $\LC(D,B)$ depending strongly continuous on $x$,
 which is also uniformly bounded and strongly continuous in $\LC(\tilde D,D)$. Let all $A(x)$  generate uniformly bounded
(for $x \in \R$ and $t$ from any compact segment) strongly continuous semigroups $e^{tA(x)}$ in $B$ with the common core $D$,
which represent also uniformly bounded strongly continuous semigroups in $D$ with the common core $\tilde D$,
and uniformly bounded semigroups in $\tilde D$.
Then the operators
\[
e^{tA(.)}: f(x) \mapsto e^{tA(x)}f(x)
\]
form a strongly continuous semigroup in $C_{\infty}(\R,B)$ with the invariant core $C_{\infty}(\R,D)$,
and a strongly continuous semigroup in $C_{\infty}(\R,D)$.

(ii) Assume additionally that the function $x\mapsto A(x)$ is differentiable both as the mapping $\R \to \LC(D,B)$
and as a mapping $\R \to \LC(\tilde D,D)$ and the derivatives $A'(x)$ (here by prime we denote the derivative with respect to $x$)
represent uniformly (in $x$) bounded and strongly continuous families
of operators again both in  $\LC(D,B)$ and $\LC(\tilde D,D)$. Then the operators $e^{tA(.)}$ represent a strongly continuous semigroup
in the Banach space $C^1_{\infty}(\R,B)\cap C_{\infty}(\R,D)$ (with the norm defined as the sum of the norms in
$C^1_{\infty}(\R,B)$ and  $C_{\infty}(\R,D)$) with the invariant core $C^1_{\infty}(\R,D)\cap C_{\infty}(\R,\tilde D)$.
Reduced to the latter space, the operators $e^{tA(.)}$ form a semigroup of bounded operators, if this space is equipped
with its own Banach topology.

(iii) Under the assumptions (i) and (ii) let $e^{tA(x)}$ have common types of growth, $m_0^B$ and $m_0^D$, $\tilde m_0^D$, as the semigroups in
$B$, $D$ and $\tilde D$ respectively, so that
 \begin{equation}
\label{eq0thLevyFellersubBvaltimedep}
\|e^{tA(x)}\|_{B\to B} \le M_B e^{tm_B},
\quad
\|e^{tA(x)}\|_{D\to D} \le M_D e^{tm_D},
\quad
\|e^{tA(x)}\|_{\tilde D\to \tilde D} \le \tilde M_D e^{t \tilde m_D},
\end{equation}
 for any $m_B>m_0^B$,  $m_D>m_0^D$, $\tilde m_D > \tilde m_0^D$ and some $M_B, M_D, \tilde M_D$. Then
 the semigroup $e^{tA(.)}$ in $C^1_{\infty}(\R,B)\cap C_{\infty}(\R,D)$ has the type of growth not exceeding $\max(m_0^B,m_0^D)$
 and enjoy the estimates
\[
\|e^{tA(.)}\|_{\LC(C^1_{\infty}(\R,B)\cap C_{\infty}(\R,D))}
\le \max(M_B e^{tm_B}, M_D e^{tm_D}+M_B M_D t e^{t \max(m_B,m_D)} \sup_x \|A'(x)\|_{D\to B})
\]
\begin{equation}
\label{eq00thLevyFellersubBvaltimedep}
\le \max(M_D,M_B) \exp\{t(\max(m_B,m_D)+M_B \sup_x \|A'(x)\|_{D\to B})\};
\end{equation}
in $C^1_{\infty}(\R,B)\cap C_{\infty}(\R,D)$ this semigroup has the type of growth not exceeding $\max(m_0^D, \tilde m_0^D)$
and enjoy the estimates
\[
\|e^{tA(x)}\|_{\LC(C^1_{\infty}(\R,D)\cap C_{\infty}(\R,\tilde D))}
\le \max(M_D e^{tm_D}, \tilde M_D e^{t \tilde m_D}+M_D \tilde M_D  t e^{t \max(m_D, \tilde m_D)} \sup_x \|A'(x)\|_{\tilde D\to D})
\]
\begin{equation}
\label{eq000thLevyFellersubBvaltimedep}
\le \max(M_D, \tilde M_D) \exp\{t(\max(m_D, \tilde m_D)+\tilde M_D \sup_x \|A'(x)\|_{\tilde D\to D})\}.
\end{equation}
\end{theorem}

\begin{proof}
(i) Since
\[
e^{tA(x)}f(x)-e^{tA(x_0)}f(x_0)=e^{tA(x)}(f(x)-f(x_0))+(e^{tA(x)}-e^{tA(x_0)})f(x_0),
\]
$e^{tA(x)}f(x)$ belongs to  $C_{\infty}(\R,B)$ (resp.  $C_{\infty}(\R,D)$)
whenever $f$ does, so that the operators $e^{tA(.)}$ represent semigroups both in $C_{\infty}(\R,B)$ and $C_{\infty}(\R,D)$.
By uniform boundedness of $e^{tA(x)}$ with respect to $x$, these semigroups are locally bounded (bounded for $t$ from compact segments).

Next, for $f\in C_{\infty}(\R,D)$,
\[
e^{tA(x)}f(x)-f(x)=\int_0^t A(x)e^{sA(x)}f(x)\, ds,
\]
which tends to zero in $B$, as $t\to 0$, because $A(x)$ and $e^{tA(x)}$ are uniformly bounded as operators
from $\LC(D,B)$ and $\LC(D,D)$ respectively. By the density argument and the boundedness of the operators
$e^{tA(x)}$ in $\LC(B,B)$, it implies the strong continuity of the semigroup $e^{tA(.)}$ in $C_{\infty}(\R,B)$.

Similarly, for $f\in C_{\infty}(\R,\tilde D)$, $e^{tA(x)}f(x)-f(x)\to 0$ in $D$, as $t\to 0$, because $A(x)$
and $e^{tA(x)}$ are uniformly bounded as operators from $\LC(\tilde D,D)$ and $\LC(\tilde D,\tilde D)$ respectively.
By the boundedness of the operators $e^{tA(x)}$ in $\LC(D,D)$, it implies the strong continuity of the semigroup
$e^{tA(.)}$ in $C_{\infty}(\R,D)$.

It remains to show that any $f\in C_{\infty}(\R,D)$ belongs to the domain of the generator $A(.)$
 of the semigroup $e^{tA(.)}$ in $C_{\infty}(\R,B)$. We have
\[
\frac{1}{t} (e^{tA(x)}f(x)-f(x))=A(x)f(x)+ \frac{1}{t}\int_0^t A(x) (e^{sA(x)}-1) f(x)\, ds,
\]
and the second term tends to zero, as $t\to 0$, due to the strong continuity of $e^{sA(.)}$ in $C_{\infty}(\R,B)$.

(ii) Since
\[
\frac{d}{dx}[e^{tA(x)}f(x)]= \lim_{\de \to 0}\frac{1}{\de}[(e^{tA(x+\de)}-e^{tA(x)}f(x)+e^{tA(x+\de)}(f(x+\de)-f(x))],
\]
 we derive that, if $f\in C^1_{\infty}(\R,B)\cap C_{\infty}(\R,D)$,
\begin{equation}
\label{eq1thLevyFellersubBvaltimedep}
\frac{d}{dx}[e^{tA(x)}f(x)]=\int_0^t e^{(t-s)A(x)} A'(x) e^{sA(x)} f(x) \, ds +e^{tA(x)}f'(x)
\end{equation}
is well defined in the topology of $B$ and represents an element of $C_{\infty}(\R,B)$, because $A'(x)$ is assumed to
be bounded and strongly continuous as a family in $\LC(D,B)$.
By the strong continuity of $e^{sA(.)}$ in $C_{\infty}(\R,B)$, it follows that
\[
\frac{d}{dx}[e^{tA(x)}f(x)]\to f'(x)
\]
as $t\to 0$. But by (i), the operators $e^{sA(.)}$ depend strongly continuous on $s$ in  $C_{\infty}(\R,D)$. Consequently,
$e^{sA(.)}$ form a strongly continuous semigroup in $C^1_{\infty}(\R,B)\cap C_{\infty}(\R,D)$.

If $f\in C^1_{\infty}(\R,D)\cap C_{\infty}(\R,\tilde D)$, then formula \eqref{eq1thLevyFellersubBvaltimedep} holds also in the topology of
 $C_{\infty}(\R,D)$ implying that $e^{sA(.)}$ preserve the space
$C^1_{\infty}(\R,D)\cap C_{\infty}(\R,\tilde D)$ and act as bounded operators in the Banach topology of this space.

(iii) By \eqref{eq1thLevyFellersubBvaltimedep},
\[
\|\frac{d}{dx}[e^{tA(x)}f(x)]\|_{C_{\infty}(\R,B)}
\]
\[
\le  M_B e^{tm_B}\|f(x)\|_{C^1_{\infty}(\R,B)}
+M_B M_D t e^{(t-s)m_B} e^{sm_D} \sup_x \|A'(x)\|_{D\to B} \|f(x)\|_{C_{\infty}(\R,D)},
\]
implying the first inequality in \eqref{eq00thLevyFellersubBvaltimedep},
from which it follows that the type of growth of $e^{tA(.)}$ in
$C^1_{\infty}(\R,B)\cap C_{\infty}(\R, D)$ does not exceed $\max(m_0^B,m_0^D)$.
The last inequality in \eqref{eq00thLevyFellersubBvaltimedep} follows by
the estimate
\[
1+tM_B \sup_x \|A'(x)\|_{D\to B} \le \exp\{t M_B \sup_x \|A'(x)\|_{D\to B}\}.
\]
Similarly  \eqref{eq00thLevyFellersubBvaltimedep} is obtained from the estimate
\[
\|\frac{d}{dx}[e^{tA(x)}f(x)]\|_{C_{\infty}(\R,D)}
\]
\[
\le  M_D e^{tm_D}\|f(x)\|_{C^1_{\infty}(\R,D)}
+M_D \tilde M_D t e^{(t-s)m_D} e^{s \tilde m_D} \sup_x \|A'(x)\|_{\tilde D\to D} \|f(x)\|_{C_{\infty}(\R,D)}.
\]

\end{proof}

\begin{theorem}
\label{thLevyFellersubBvaltimedep1}
(i) Under the assumptions of Theorem \ref{thLevyFellersubBvaltimedep} (i)
let $\nu$ be a bounded measure on the ray $\{y:y>0\}$. Then the operator $L'_{\nu}+A(.)$ generates a
strongly continuous semigroup $\Phi_t^{\nu,A}$ in $C_{\infty}(\R,B)$ with the invariant core $C_{\infty}(\R,D)$,
where this semigroup is also strongly continuous. The semigroup $\Phi_t$ has the representation:
\[
 \Phi_t^{\nu,A} Y(x)=e^{-t\|\nu\|} \bigl[ e^{tA(x)}Y(x)
 +\sum_{m=1}^{\infty} \int_{0\le s_1\le \cdots \le s_m\le t} ds_1 \cdots ds_m \, \nu(dz_1) \cdots \nu (dz_m)
 \]
\begin{equation}
\label{eq2thLevyFellersubBvaltimedep1}
\times \exp\{ s_1A(x)\} \exp\{ (s_2-s_1)A(x-z_1)\} \cdots \exp\{ (t-s_m) A(x-z_1-\cdots - z_m) \} Y(x-z_1-\cdots - z_m) \bigr],
 \end{equation}
or, using notation \eqref{eqgenpathforjumpproc} for piecewise-continuous paths,
\[
 \Phi_t^{\nu,A} Y(x)=e^{-t\|\nu\|} \bigl[ e^{tA(x)}Y(x)
 +\sum_{m=1}^{\infty} \int_{0\le s_1\le \cdots \le s_m\le t} ds_1 \cdots ds_m \, \nu(dz_1) \cdots \nu (dz_m)
 \]
\begin{equation}
\label{eq1thLevyFellersubBvaltimedep1}
\times \exp\{\int_0^{s_1} A(Z_x(\tau)) d\tau \}
\exp\{ \int_{s_1}^{s_2} A(Z_x(\tau)) d\tau\}
 \cdots \exp\{ \int_{s_m}^t A(Z_x(\tau)) d\tau \} Y(Z_x(t))\bigr].
 \end{equation}

(ii) For any $b>a$, the operators $\Phi_t^{\nu,A}$ represent strongly continuous semigroups also in the spaces
$C_{kill(a)}([a,b],B)$ and $C_{kill(a)}([a,\infty),B)$ (the latter is a closed subspace of
 $C_{\infty}(\R,B)$, but the former is not).

(iii) If the assumptions of Theorem \ref{thLevyFellersubBvaltimedep} (ii) hold, then
 $L'_{\nu}+A(.)$ generates also a strongly continuous semigroup in $C^1_{\infty}(\R,B)\cap C_{\infty}(\R,D)$
 with the invariant core $C^1_{\infty}(\R,D)\cap C_{\infty}(\R,\tilde D)$. The operators $\Phi_t^{\nu,A}$ are
 bounded in the space $C^1_{\infty}(\R,D)\cap C_{\infty}(\R,\tilde D)$ equipped with its own Banach topology.

\end{theorem}

\begin{proof}
(i) Since $L'_{\nu}$ is a bounded operator both in $C_{\infty}(\R,B)$ and in $C_{\infty}(\R,D)$,
it follows from the perturbation theory that the operator
\[
(L'_{\nu}+A(.))f(x)=\int f(x-y) \nu (dy) +(A(x)-\|\nu\|)f(x)
\]
generates a
strongly continuous semigroup in $C_{\infty}(\R,B)$ with the invariant core $C_{\infty}(\R,D)$,
where this semigroup is also strongly continuous. Moreover, formula \eqref{eqperturbseries} (with
the operator $\int f(x-y) \nu (dy)$ considered as a bounded perturbation) provides representation
\eqref{eq2thLevyFellersubBvaltimedep1}. Unlike \eqref{eqperturbseriespathint} the operators $A(x)$
may not commute and thus the exponents can not be put together. Due to notations \eqref{eqgenpathforjumpproc},
equations \eqref{eq1thLevyFellersubBvaltimedep1} and \eqref{eq2thLevyFellersubBvaltimedep1} are equivalent.

(ii) The invariance of the spaces $C_{kill(a)}([a,b],B)$ and $C_{kill(a)}([a,\infty),B)$ under $\Phi_t^{\nu,A}$ is seen
from \eqref{eq2thLevyFellersubBvaltimedep1}.

(iii)This follows again by the perturbation theory and the observation that the operators
$e^{tA(.)}$ and $L'_{\nu}$ are bounded in  the space $C^1_{\infty}(\R,D)\cap C_{\infty}(\R,\tilde D)$ equipped with its own Banach topology.
\end{proof}

Recall that the product of exponents in  \eqref{eq1thLevyFellersubBvaltimedep1} or \eqref{eq2thLevyFellersubBvaltimedep1}
is called the (backward) chronological or time-ordered exponential (or $T$-product) that is usually denoted $T\exp\{\int_0^t A(Z_x(\tau)) \, d\tau \}$
(we use the letter $T$ for the backward exponentials, as forward exponentials will not be used here at all).

Denoting by $\nu^{PC}$ the measure on $PC_x(t)$ constructed from $\nu$
we can rewrite   \eqref{eq1thLevyFellersubBvaltimedep1} as

\[
 \Phi_t^{\nu,A} Y(x)=e^{-t\|\nu\|} \int_{PC_x(t)} F(Z_x(.)) \nu^{PC}(dZ(.))
\]
with
\[
F(Z_x(.))= T\exp\{\int_0^t A(Z_x(\tau)) \, d\tau \}Y(Z_x(t)).
\]
Introducing the normalized probability measure $\tilde \nu^{PC}=e^{-t\|\nu\|}\nu^{PC}$ and denoting by
$\E_{\nu}$ (the expectation) the integration with respect to this measure on the path-space $PC_x(t)$ we arrive at the main representation formula.

\begin{corollary}
\label{cor1thLevyFellersubBvaltimedep1}
Under the assumptions of Theorem \ref{thLevyFellersubBvaltimedep1},
 the semigroup $\Phi_t^{\nu,A}$ yielding the unique solution to the Cauchy problem
\begin{equation}
\label{eq3-1thLevyFellersubBvaltimedep1}
 \dot \mu_t(x)=A(x)\mu_t(x) -D^{(\nu)}_+ \mu_t(x), \quad \mu_0=Y,
 \end{equation}
 has the following integral representation in terms of the backward chronological exponential:
\begin{equation}
\label{eq3thLevyFellersubBvaltimedep1}
 \Phi_t^{\nu,A} Y(x)=\int_{PC_x(t)} F(Z_x(.)) \tilde \nu^{PC}(dZ(.))
 =\E_{\nu} [ T\exp\{\int_0^t A(Z_x(\tau)) \, d\tau \}Y(Z_x(t))].
 \end{equation}
\end{corollary}

The next consequence shows that formula \eqref{eq2thLevyFellersubBvaltimedep1} allows one to find
 the growth of the semigroup $\Phi_t^{\nu,A}$, whenever the growth of $e^{tA(x)}$ is known.

\begin{corollary}
\label{cor2thLevyFellersubBvaltimedep1}
Under the assumptions of Theorem \ref{thLevyFellersubBvaltimedep1} (i)-(iii),
\begin{equation}
\label{eq4thLevyFellersubBvaltimedep1}
\| \Phi_t^{\nu,A} \|_{\LC(C_{\infty}(\R,B))}\le M_B \exp\{ t(m_B+\|\nu\|(M_B-1)) \},
 \end{equation}
 \begin{equation}
\label{eq4athLevyFellersubBvaltimedep1}
\| \Phi_t^{\nu,A} \|_{\LC(C_{\infty}(\R,D))}\le M_D \exp\{ t(m_D+\|\nu\|(M_D-1)) \},
 \end{equation}
\[
\| \Phi_t^{\nu,A} \|_{\LC(C^1_{\infty}(\R,B)\cap C_{\infty}(\R,D))}
\le \max(M_B,M_D)
\]
\begin{equation}
\label{eq5thLevyFellersubBvaltimedep1}
\exp\{ t[\max(m_D,m_B)+M_B \sup_x \|A'(x)\|_{D\to B}+\|\nu\|(\max(M_B,M_D)-1)] \},
 \end{equation}
\[
\| \Phi_t^{\nu,A} \|_{\LC(C^1_{\infty}(\R,D)\cap C_{\infty}(\R,\tilde D))}
\le \max(M_D, \tilde M_D)
\]
\begin{equation}
\label{eq5athLevyFellersubBvaltimedep1}
\exp\{ t[\max(m_D,\tilde m_D)+M_D \sup_x \|A'(x)\|_{\tilde D\to D}+\|\nu\|(\max(M_D, \tilde M_D)-1)] \},
 \end{equation}

In particular, if the semigroups $e^{tA(x)}$ are regular in $B$ and $D$ in the sense that
\eqref{eq0thLevyFellersubBvaltimedep} holds with $M_D=M_B=1$ and some $m_D,m_B$,
which is equivalent to the requirement that
\begin{equation}
\label{eqregboun}
\sup_t \sup_x \frac{1}{t} \ln \|e^{tA(x)}\|_{\LC(B)} <\infty,
\quad
\sup_t \sup_x \frac{1}{t} \ln \|e^{tA(x)}\|_{\LC(D)} <\infty,
 \end{equation}
then so is the semigroup $\Phi_t^{\nu,A}$ both in $C_{\infty}(\R,B)$ and
$C^1_{\infty}(\R,B)\cap C_{\infty}(\R,D)$, and its growth rates are given by the estimates
\begin{equation}
\label{eq6thLevyFellersubBvaltimedep1}
\| \Phi_t^{\nu,A} \|_{\LC(C_{\infty}(\R,B))}\le \exp\{ t m_B \},
\quad
\| \Phi_t^{\nu,A} \|_{\LC(C_{\infty}(\R,D))}\le \exp\{ t m_D \},
 \end{equation}
\begin{equation}
\label{eq7thLevyFellersubBvaltimedep1}
\| \Phi_t^{\nu,A} \|_{\LC(C^1_{\infty}(\R,B)\cap C_{\infty}(\R,D))}
\le \exp\{ t[\max(m_D,m_B)+\sup_x \|A'(x)\|_{D\to B}] \},
 \end{equation}
independent of $\nu$.
\end{corollary}

\begin{proof}
By \eqref{eq2thLevyFellersubBvaltimedep1},
\[
\| \Phi_t^{\nu,A} \|_{C_{\infty}(\R,B)} \le M_Be^{m_Bt}(1 +\sum_{n=1}^{\infty} \|\nu\|^nM_B^n t^n/n!),
\]
implying  \eqref{eq4thLevyFellersubBvaltimedep1}. Similarly  other estimates are obtained due to
\eqref{eq00thLevyFellersubBvaltimedep} and \eqref{eq000thLevyFellersubBvaltimedep}.
\end{proof}

We can now address problem  \eqref{eqlinder1unbgentimedep} in the simplest case of bounded $\nu$.

\begin{theorem}
\label{thLevyFellersubBvaltimedep2}
Under the assumptions of Theorem \ref{thLevyFellersubBvaltimedep1},
the resolvent operators $R^{A,\nu}_{\la}$ of the semigroup $\Phi_t^{\nu,A}$ in the space $C_{kill(a)}([a,\infty),B)$
yielding the classical solutions to the problems
 \begin{equation}
\label{eq1thLevyFellersubBvaltimedep2}
(\la - A(x)+ D^{(\nu)}_{a+})\mu(x)=g(x), \quad \mu(a)=0, \quad x\ge a,
\end{equation}
are well defined for
\[
\la > m_B+\|\nu\|(M_B-1),
\]
and are given  by the formula
 \begin{equation}
\label{eq2thLevyFellersubBvaltimedep2}
R^{A,\nu}_{\la}g(x) =\int_0^{\infty} e^{-\la t} \E_{\nu} \,[ T\exp\{\int_0^t A(Z_x(\tau)) \, d\tau \}g(Z_x(t))] \, dt.
\end{equation}

When reduced to $C_{kill(a)}([a,b],B)$, they are also well defined for  $\la \ge m_B+\|\nu\|(M_B-1)$.
In particular, if all semigroups generated by $A(x)$ in $B$ are contractions,
problem \eqref{eqlinder1unbgentimedep} with $Y=0$ has a unique classical solution (belonging to the domain of the
generator of the semigroup $\Phi_t^{\nu,A}$ in $C_{kill(a)}([a,b],B)$)
given by \eqref{eq2thLevyFellersubBvaltimedep2} with $\la=0$ for any $g\in C_{kill(a)}([a,b],B)$.
\end{theorem}

Since $g\in C_{kill(a)}([a,\infty),B)$, formula \eqref{eq2thLevyFellersubBvaltimedep2} rewrites as
 \begin{equation}
\label{eq3thLevyFellersubBvaltimedep2}
R^{A,\nu}_{\la}g(x) =\E_{\nu} \int_0^{\si_a} e^{-\la t} [ T\exp\{\int_0^t A(Z_x(\tau)) \, d\tau \}g(Z_x(t))] \, dt,
\end{equation}
where $\si_a=\inf\{t: Z_x(t) \le a\}$.
This formula can be used to define various generalized solutions to \eqref{eq1thLevyFellersubBvaltimedep2}.


\section{Time-nonhomogeneous case: arbitrary $\nu$}
\label{secgenfractimedep1}

Let us turn to problems \eqref{eqlinder1unbgentimedep} with an unbounded $\nu$.

\begin{theorem}
\label{thLevyFellersubBvaltimedep3}
(i) Under the assumptions of Theorem \ref{thLevyFellersubBvaltimedep} (i)-(iii) let the semigroups
$e^{tA(x)}$ be regular in $B$ and $D$ in the sense that \eqref{eq0thLevyFellersubBvaltimedep} holds
 with $M_D=M_B=1$ and some $m_D,m_B$ (equivalently, if \eqref{eqregboun} hold). Let $\nu$ be a measure
  on the ray $\{y:y>0\}$ satisfying \eqref{eq0diftointfraceqgennu}.
Then the operator $L'_{\nu}+A(.)$ generates a strongly continuous semigroup $\Phi_t^{\nu,A}$
both in $C_{\infty}(\R,B)$ and $C_{\infty}(\R,D)$ solving the Cauchy problem
\eqref{eq3-1thLevyFellersubBvaltimedep1}, with the domains of the generator containing the spaces
 $C^1_{\infty}(\R,B)\cap C_{\infty}(\R,D)$ and $C^1_{\infty}(\R,B)\cap C_{\infty}(\R,D)$ respectively.
  The semigroup $\Phi_t^{\nu,A}$ can be obtained as the limit, as $\ep \to 0$,
of the semigroups $\Phi_t^{\nu_{\ep},A}$ built by Theorem  \ref{thLevyFellersubBvaltimedep1} for the finite approximations
$\nu_{\ep}(dy)=\1_{|y|\ge \ep}\nu(dy)$ of $\nu$, so that the semigroup $\Phi_t^{\nu,A}$ has the representation
\begin{equation}
\label{eq1thLevyFellersubBvaltimedep3}
\Phi_t^{\nu,A} Y(x)=\lim_{\ep \to 0} \Phi_t^{\nu_{\ep},A} Y(x)
 =\lim_{\ep \to 0}\E_{\nu_{\ep}} [ T\exp\{\int_0^t A(Z_x(\tau)) \, d\tau \}Y(Z_x(t))],
 \end{equation}
 where the limit is well defined both in the topologies of $B$ and $D$, and
\begin{equation}
\label{eq1athLevyFellersubBvaltimedep3}
\| \Phi_t^{\nu,A} \|_{\LC(C_{\infty}(\R,B))}\le \exp\{ t m_B \},
\quad
\| \Phi_t^{\nu,A} \|_{\LC(C_{\infty}(\R,D))}\le \exp\{ t m_D \},
 \end{equation}

(ii) For any $b>a$, the operators $\Phi_t^{\nu,A}$ represent strongly continuous semigroups also in the spaces
$C_{kill(a)}([a,b],B)$, $C_{kill(a)}([a,\infty),B)$, $C_{kill(a)}([a,b],D)$, $C_{kill(a)}([a,\infty),D)$.
\end{theorem}

\begin{proof}

(i) It is similar to the proof of Proposition  \ref{propLevyFellersubBval}.
By \eqref{eq6thLevyFellersubBvaltimedep1} and \eqref{eq7thLevyFellersubBvaltimedep1} the semigroups
$\Phi_t^{\nu_{\ep},A}$ are uniformly (in $\ep$) bounded in both $C_{\infty}(\R,B)$ and
$C^1_{\infty}(\R,B)\cap C_{\infty}(\R,D)$. Estimating the difference between the actions of
$\Phi_t^{\nu_{\ep},A}$ for the two values $\ep_2<\ep_1<1$ in the usual way we get

\begin{equation}
\label{eq1thLevyFellersubBvaltimedep3}
\Phi_t^{\nu_{\ep_1},A}-\Phi_t^{\nu_{\ep_2},A}
=\int_0^t \Phi_{t-s}^{\nu_{\ep_2},A}(L_{\nu_{\ep_1}}-L_{\nu_{\ep_1}}) \Phi_s^{\nu_{\ep_1},A} \, ds.
\end{equation}
Hence, for $Y\in C^1_{\infty}(\R,B)\cap C_{\infty}(\R,D)$, we derive by \eqref{eq6thLevyFellersubBvaltimedep1}
 and \eqref{eq7thLevyFellersubBvaltimedep1} that
\[
\|(\Phi_t^{\nu_{\ep_1},A}-\Phi_t^{\nu_{\ep_2},A})Y\|_{C_{\infty}(\R,B)}
\]
\[
\le \int_0^t ds \, e^{(t-s)m_B}
\sup_x \left\|\int_{\ep_2}^{\ep_1} (\Phi_s^{\nu_{\ep_1},A}Y(x-y)-\Phi_s^{\nu_{\ep_1},A}Y(x))\nu(dy)\right\|_B
\]
\[
\le  \int_0^t ds \, e^{(t-s)m_B}\int_{\ep_2}^{\ep_1} y\nu(dy) \|\Phi_s^{\nu_{\ep_1},A}Y\|_{C^1_{\infty}(\R,B)},
\]
and thus
\[
\|(\Phi_t^{\nu_{\ep_1},A}-\Phi_t^{\nu_{\ep_2},A})Y\|_{C_{\infty}(\R,B)}
\le \int_{\ep_2}^{\ep_1} y\nu(dy)
\]
\begin{equation}
\label{eq2thLevyFellersubBvaltimedep3}
\times \int_0^t ds \, e^{(t-s)m_B} \exp\{ s[\max(m_D,m_B)+\sup_x \|A'(x)\|_{D\to B}] \} \|Y\|_{C^1_{\infty}(\R,B)\cap C_{\infty}(\R,D)},
\end{equation}
which tends to zero as $\ep_1,\ep_2 \to 0$ uniformly for $t$ from any compact set. Hence
the families $\Phi_t^{\nu_{\ep_2},A}Y$ converge, as $\ep\to 0$, for any $Y\in C^1_{\infty}(\R,B)\cap C_{\infty}(\R,D)$.
By the density argument this convergence extends to all $Y\in C_{\infty}(\R,B)$. Passing to the limit in the semigroup equation
 we derive that the limiting operators form a bounded semigroup in $C_{\infty}(\R,B)$, with the same bounds \eqref{eq6thLevyFellersubBvaltimedep1}.
 We denote this semigroup $\Phi_t^{\nu,A}$. Its strong continuity follows from the  strong continuity of $\Phi_t^{\nu_{\ep},A}$.

 Writing
 \[
 \frac{\Phi_t^{\nu,A}Y-Y}{t}=\frac{\Phi_t^{\nu_{\ep},A}Y-Y}{t}+\frac{\Phi_t^{\nu,A}Y-\Phi_t^{\nu_{\ep},A}Y}{t}
 \]
 and noting that, by \eqref{eq2thLevyFellersubBvaltimedep3}, the second term tends to zero, as $t,\ep \to 0$, we can conclude that
the space $C^1_{\infty}(\R,B)\cap C_{\infty}(\R,D)$ belongs to the domain of the semigroup $\Phi_t^{\nu,A}$ in $C_{\infty}(\R,B)$.

Finally, the same estimates as above can be performed in the space topology of $C_{\infty}(\R,D)$ showing the required properties
of $\Phi_t^{\nu,A}$ in $C_{\infty}(\R,D)$. Estimates \eqref{eq1athLevyFellersubBvaltimedep3} follow from the same estimate for
$\Phi_t^{\nu_{\ep},A}$.

(ii) The invariance of the spaces $C_{kill(a)}([a,b],B)$, $C_{kill(a)}([a,\infty),B)$, $C_{kill(a)}([a,b],D)$ and
 $C_{kill(a)}([a,\infty),D)$ under $\Phi_t^{\nu,A}$ follows from their invariance under all $\Phi_t^{\nu_{\ep},A}$.
 \end{proof}

From the point of view of numeric calculations, the limiting integral representation formula  \eqref{eq1thLevyFellersubBvaltimedep3}
seems to be most appropriate. Theoretically it is of course desirable to get rid of $\lim_{\ep \to 0}$.

\begin{theorem}
\label{thLevyFellersubBvaltimedep4}
Under the assumptions of Theorem \ref{thLevyFellersubBvaltimedep2}
formula  \eqref{eq3thLevyFellersubBvaltimedep1}
can be represented in the equivalent form
\begin{equation}
\label{eqtimeordFKfo}
\Phi_t^{\nu,A} Y(x)=\E_{\nu} [ T\exp\{\int_0^t A(Z_x(\tau)) \, d\tau \}Y(Z_x(t))],
 \end{equation}
 where $\E_{\nu}$ here means the expectation with respect to the measure on the cadlag paths of the L\'evy process
 generated by the operator $L'_{\nu}$ and started at $x$.
 \end{theorem}

 \begin{proof}
This follows from \eqref{eq1thLevyFellersubBvaltimedep2} and three additional points: (i) convergence of Feller semigroups implies the weak
convergence of the corresponding Markov processes, (ii) the limiting process generated by $L'_{\nu}$ is a L\'evy processes, whose trajectories are
non-increasing cadlad paths, (iii) the convergence of propagators parametrized by cadlag paths, see Theorem 1.9.5 of \cite{Ko11}.
 \end{proof}

Formula \eqref{eqtimeordFKfo} is a performance of the time-ordered operator-valued Feynman-Kac formula of stochastic calculus.

As a consequence, like in the case of bounded $\nu$, we obtain the solutions to problem \eqref{eqlinder1unbgentimedep}.

\begin{theorem}
\label{thLevyFellersubBvaltimedep5}
Under the assumptions of Theorem \ref{thLevyFellersubBvaltimedep2},
the resolvent operators $R^{A,\nu}_{\la}$ of the semigroup $\Phi_t^{\nu,A}$ in the space $C_{kill(a)}([a,\infty),B)$
yielding the classical solutions to problems \eqref{eq1thLevyFellersubBvaltimedep2}
are well defined for $\la > m_B$ and are given  by the formula
\[
R^{A,\nu}_{\la}g(x)
=\E_{\nu}  \int_0^{\si_a} e^{-\la s} [ T\exp\{\int_0^t A(Z_x(\tau)) \, d\tau \}g(Z_x(t))] \, ds
\]
\begin{equation}
\label{eqtimeordFKfost}
=\lim_{\ep\to 0} \E_{\nu_{\ep}} \int_0^{\si_a} e^{-\la s} [ T\exp\{\int_0^s A(Z_x(\tau)) \, d\tau \}g(Z_x(t))] \, ds.
\end{equation}
If all semigroups generated by $A(x)$ in $B$ are contractions, problem \eqref{eqlinder1unbgentimedep} with $Y=0$ has
a unique classical solution (belonging to the domain of the generator of the semigroup $\Phi_t^{\nu,A}$ in
$C_{kill(a)}([a,b],B)$) for any $g\in C_{kill(a)}([a,b],B)$.
\end{theorem}


As usual, formula \eqref{eqtimeordFKfost} yields also the generalized solutions,
by approximations or duality, to problems \eqref{eq1thLevyFellersubBvaltimedep2},
if $g$ is any bounded measurable function $[a,\infty) \to B$.

Again as usual, one defines solutions to problem \eqref{eqlinder1unbgentimedep} with arbitrary $Y$,
by shifting, that is, as the function $\mu(x)=Y+u(x)$, where $u$ solves the problem
 \begin{equation}
\label{eqCapgenshift}
D^{(\nu)}_{a+*}u(x)=A(x) u(x)+A(x)Y +g(x), \quad \mu(a)=0, \quad x\ge a.
\end{equation}
This leads to the following.

\begin{corollary}
We conclude that under the assumptions of Theorem \ref{thLevyFellersubBvaltimedep2},
if all semigroups generated by $A(x)$ in $B$ are contractions, problem \eqref{eqlinder1unbgentimedep} has
the unique generalized solution
\begin{equation}
\label{eqtimeordFKfost1}
\mu(x)=Y+\E_{\nu} \int_0^{\si_a} [ T\exp\{\int_0^s A(Z_x(\tau)) \, d\tau \}(A(Z_x(t))Y+g(Z_x(t)))] \, ds,
\end{equation}
for any $Y\in D$ and a bounded measurable curve $g: [a,\infty) \to B$.
\end{corollary}

\begin{remark}
(i) Assuming some regularity on $\nu$, like in Proposition \ref{generlinfraccomp1}
one can relax the assumption of $A(x)$. Various additional regularity properties of solutions can be obtained
by assuming some smoothing properties of the semigroups $e^{tA(x)}$.

(ii) Assuming the existence of the bounded second derivative $A''(x)$ would allow one to show that the space
$C^1_{\infty}(\R,B)\cap C_{\infty}(\R,D)$ represents an invariant core for $\Phi^{\nu,A}_t$.

(iii) Backward time-ordered exponential $T\exp\{\int_s^t A(Z_x(\tau)) \, d\tau \}$ represents the
 backward propagator $U^{s,t}$ solving the backward Cauchy problem
\begin{equation}
\label{eqtimeordFKfost2}
\dot f_s(x)= - A(Z_x(s))f_s(x), \quad s\le t,
\end{equation}
with the given terminal condition $f_t$, where the family  $A(Z_x(t))$ is bounded (as operators $D\to B$),
but discontinuous in $t$. However, by the property of L\'evy processes, it has at most countable discontinuity-set.
\end{remark}

\section{Basic examples}
\label{secexam}

Let us present some examples, when basic formula  \eqref{eqtimeordFKfost1} is applicable.
For better fit to the customary notations, we shall use the letter $t$ for the argument, rather than $x$ used above,
where $t$ was used as the time variable in the auxiliary semigroups.

(i) Generalized {\it fractional Schr\"odinger equation}\index{fractional Schr\"odinger equation}
with time-dependent Hamiltonian and generalized fractional derivative:
\begin{equation}
\label{eqdefgenfracSchrext1}
D^{(\nu)}_{a+*} \psi_t=-i H(t) \psi_t,
\end{equation}
where $H(t)$ is a family of self-adjoint operators in a Hilbert space $\HC$ such that the unitary groups generated by $H(t)$ have
 a common domain $D\subset H$, where they are regular in the sense of the second condition of \eqref{eqregboun}.

 The simplest concrete example represent the Hamiltonians $H(t)=-\De +V(t,x)$ with $V(t,.)\in C^2(\R^d)$,
 where $D$ can be chosen as the Sobolev space $H^2_2(\R^d)$.

Similarly one can deal with {\it fractional Schr\"odinger equation}\index{fractional Schr\"odinger equation} with the complex parameter
\begin{equation}
\label{eqdefgenfracSchrext2}
D^{(\nu)}_{a+*} \psi_t=\si H(t) \psi_t,
\end{equation}
if $H$ is a negative operator and $\si$ is a complex number with a non-negative real part,
and where again a common domain $D\subset H$ exists such that the semigroups generated by $\si H(t)$ are regular.
 In both cases, formula  \eqref{eqtimeordFKfost1} is applicable.
Specific examples of these equations were analyzed recently in \cite{GorSchrod17}.

(ii) Generalized {\it fractional Feller evolution}\index{fractional Feller evolution},
where each $A(t)$ in \eqref{eqCapgenshift} generates a Feller semigroup in $C_{\infty}(\R^d)$, again with the additional property that
the semigroups generated by $A(t)$ act regularly in their invariant cores $D$ that can be often taken as $C^1_{\infty}(\R^d)$
(for operators of at most first order) or $C^2_{\infty}(\R^d)$ (e.g. for diffusions).

(iii) Generalized fractional evolutions generated by $\Psi$DOs with spatially homogeneous symbols (or with constant coefficients):
\begin{equation}
\label{eqlinpsidoconstfrreprep}
D^{(\nu)}_{a+*} f_t=-\psi_t(-i\nabla) f_t +g_t, \quad f|_{t=a}=f_a,
\end{equation}
under various assumptions on symbols $\psi_t(p)$ ensuring that $-\psi_t(-i\nabla)$ generates a semigroup.
 In this case propagators solving \eqref{eqtimeordFKfost2}
are constructed explicitly via the Fourier transform.
For instance, formula  \eqref{eqtimeordFKfost} for the solution of \eqref{eqlinpsidoconstfrreprep} with $f_a=0$ becomes
\begin{equation}
\label{eqtimeordFKfostPsido}
R^{A,\nu}_0g(t,w)
=\E_{\nu} \int_0^{\si_a} \int_{\R^d} G^{\psi,Z}_{s,0}(w-v) g_{Z_t(s)}(v) \, dv  \, ds,
\end{equation}
where
\begin{equation}
\label{eqtimeordFKfostPsido1}
G^{\psi,Z}_{s,0}(w) =\frac{1}{(2\pi)^d} \int e^{ipw} \exp \{ -\int_0^s \psi_{Z_t(\tau)} (p)\, d\tau \} \, dp.
\end{equation}

\end{document}